\documentclass[11pt,twoside, a4paper, english, reqno]{amsart}
 \usepackage[foot]{amsaddr}
\usepackage[dvips]{epsfig}
\usepackage{amscd}
\usepackage{amssymb}
\usepackage{amsthm}
\usepackage{amsmath}
\usepackage{mathtools}
\usepackage{latexsym}

\usepackage{upref}
\usepackage{bm}
\usepackage{color}
\usepackage{hyperref}
\usepackage{graphicx}
\usepackage{lipsum}
\usepackage{comment}
\usepackage{romannum}

\hypersetup{linkcolor=blue, colorlinks=true
,citecolor = red}

\theoremstyle{plain}
\newtheorem{Th}{Theorem}[section]
\newtheorem{Lem}[Th]{Lemma}
\newtheorem{Cor}[Th]{Corollary}
\newtheorem{Prop}[Th]{Proposition}

\theoremstyle{definition}

\newtheorem{Rem}[Th]{Remark}

\def\@setemails{%
  \ifnum\theg@author > 1
\mbox{{\itshape E-mail addresses}:\space}{\ttfamily\emails}. \else\mbox{{\itshape E-mail address}:\space}{\ttfamily\emails}. \fi%
}

\DeclareMathOperator{\sech}{sech}

\newcommand{\be} {\begin{equation}}
\newcommand{\ee} {\end{equation}}

\newcommand{\di}{\mbox{dist}}
\newcommand{\m}{\mathfrak{m}}
\newcommand{\ga}{\gamma}

\newcommand{\cbe}{C_{ \mbox{\tiny{BE}}}}
\newcommand{\cbel}{C_{\tiny \mbox{BE}}^{\tiny\mbox{loc}}}
\newcommand{\cber}{C_{\tiny \mbox{BE}}^{\tiny\mbox{rad}}}

\newcommand{\ca}{\mathcal{C}_{\alpha}}
\newcommand{\sg}{S_{\gamma}}
\newcommand{\mg}{\mu^{\gamma}}

\newcommand{\mhs}{\mathcal{M}_{\mbox{\tiny HS}}}

\newcommand{\ug}{U_{\gamma}}

\newcommand{\hug}{\hat{U}_{\gamma}}

\newcommand{\la} {\lambda}
\newcommand{\R} {\mathbb{R}}
\newcommand{\rn}{\R^{N}}
\newcommand{\s} {\mathbb{S}}
\newcommand{\sn}{\s^{N-1}}

\newcommand{\var}{\varepsilon}
\newcommand{\cgn}{C_{\gamma,N}}
\newcommand{\bp}{\beta_{+}(\gamma)}
\newcommand{\hsb}{\frac{\cgn}{\left(|x|^{\frac{2 \bp}{N-2}}+|x|^{\frac{2\beta_{-}(\gamma)}{N-2}}\right)^{\frac{N-2}{2}}}}
\newcommand{\ds}{{\rm d}\sigma}
\newcommand{\intR}{\int_{\R}}
\newcommand{\intr}{\int_{\rn}}
\newcommand{\intc}{\int_{\R \times \sn}}

\newcommand{\lts}{L^{2^{\star}}}
\newcommand{\lt}{L^2}

\newcommand{\authorfootnotes}{\renewcommand\thefootnote{\@fnsymbol\c@footnote}}%

\def\e{{\text{e}}}

\numberwithin{equation}{section} \allowdisplaybreaks

\title[Bianchi-Egnell extremizer for HS]{Existence of  Bianchi-Egnell stability extremizer for the Hardy-Sobolev inequality}

\author[S. Chakraborty]{Souptik Chakraborty}
\author[M. Ghosh]{Monideep Ghosh}

\author[D. Karmakar]{Debabrata Karmakar}

\email{\tt  souptik25@tifrbng.res.in,  monideep@tifrbng.res.in, debabrata@tifrbng.res.in}

 \address[]{Tata Institute of Fundamental Research, Centre For Applicable Mathematics}
 \address[]{Post Bag No 6503, GKVK Post Office,
Sharada Nagar, Chikkabommsandra,
Bangalore 560065,
Karnataka, India}



\begin{document}
\pagenumbering{arabic}
\begin{abstract}
In this article, we prove the best Bianchi-Egnell constant for the Hardy-Sobolev (HS) inequality
\begin{align*}
\cbe(\gamma) \coloneqq \inf_{\tiny{u \ \mbox{not an optimizer}}} \frac{
  \int_{\mathbb{R}^n} \left(|\nabla u|^2 - \frac{\gamma}{|x|^2}u^2\right) \ {\rm d}x - S_{\gamma}\|u\|_{L^{2^{\star}}}^2}{\mbox{dist} (u, \ \mbox{set of optimizers})^2},
\end{align*}
is attained, extending the result of K\"onig \cite{KT-2025} for the classical Sobolev inequality (that corresponds to $\gamma = 0$). One of the main difficulties is that the third eigenspace of the linearized operator may contain only spherical harmonics of degree $1$, and hence, an essential non-vanishing criterion fails \cite{KT-2023}. This non-vanishing criterion is indispensable for proving the best Bianchi-Egnell constant $C_{\mbox{\tiny{BE}}}(\gamma) < C_{\mbox{\tiny{BE}}}^{\mbox{\tiny{loc}}}(\gamma)$ that prevents a minimizing sequence converging to one of the optimizers. In addition, not being translation invariant, extracting a non-zero weak limit from a minimizing sequence presents difficulties. We found another hidden critical level $C_{\mbox{\tiny{BE}}}(\gamma) <1 - \frac{S_{\gamma}}{S},$ where $S$ is the best Sobolev constant that plays a significant role in proving the existence of an extremizer.
In particular, we show that there exists a $\gamma_0>0$ such that for $\gamma \geq \gamma_0,$
 $C_{\mbox{\tiny{BE}}}(\gamma)$ is attained. Moreover, we remark that there is a region $\gamma_0 \leq \gamma < \gamma_c^{\star},$ where the third eigenspace of the linearized operator contains only spherical harmonics of degree $1.$ Our result improves some of the results in Wei-Wu \cite{WeiWu-2024} and Deng-Tian and Wei \cite{Wei25} corresponding to the HS inequality.
 \end{abstract}

\maketitle
{\small\tableofcontents}

\section{Introduction}

The classical Hardy-Sobolev inequality in $\mathbb{R}^N$ asserts that for every $0 <\gamma < \frac{(N-2)^2}{4},$ there exists an optimal constant $S_{\gamma} = S(\gamma, N)$ such that the inequality
\begin{align} \label{HS}
S_{\gamma}\|u\|_{L^{2^{\star}}}^2 \leq \int_{\mathbb{R}^N} \left(|\nabla u|^2 - \frac{\gamma}{|x|^2}u^2\right) \ {\rm d}x, \ \ \ \ 2^{\star} = \frac{2N}{N-2}
\end{align}
holds for every $u \in H^1(\mathbb{R}^N).$ Here and throughout the article $\|u\|_{L^p}$ denotes the $L^p$-norm of $u$ on $\mathbb{R}^N$ with respect to the Lebesgue measure and $H^1(\mathbb{R}^N)$ is defined by the closure of $C_c^{\infty}(\R^N)$ with respect to the norm $\|\nabla u\|_{L^2}$. The quantity $\frac{(N-2)^2}{4}$ is the best Hardy constant. The inequality \eqref{HS} is the interpolation of the classical Sobolev inequality that corresponds to limiting case $\gamma = 0$ with best constant $S = S(\R^N)>0$, and the classical Hardy-inequality corresponding to $S_{\gamma} = 0$ and $\gamma = \frac{(N-2)^2}{4}.$ The inequality \eqref{HS} and all the quantities involved are invariant under the dilations $\lambda^{\frac{N-2}{2}}u(\lambda x).$ It is known from \cite{CC-1993,TS-1996} that for $\gamma >0,$ the equality holds in \eqref{HS} if and only if $u$ belongs to 
\begin{align}
\mathcal{M}_{\mbox{\tiny{HS}}} \coloneqq \{ c U_{\gamma}[\lambda]: c \in \mathbb{R}\backslash \{0\}, \la >0\}, 
\end{align}
where 
\begin{align} \label{HS bubble}
U_{\gamma}(x) =\hsb, \ \ \beta_{\pm}(\gamma) = \frac{N-2}{2} \pm \sqrt{\frac{(N-2)^2}{4} - \gamma},
\end{align}
$\cgn$  is a normalizing constant $\|\ug\|_{L^{2^{\star}}}=1,$ and $U_{\gamma}[\lambda](x) = \lambda^{\frac{N-2}{2}}U_{\gamma}(\lambda x).$
Hence the set of all extremizers $\mathcal{M}_{\mbox{\tiny{HS}}}$ is a two dimensional set of parameters. On the other hand, in the limiting case $\gamma = 0,$ the Sobolev inequality is invariant under dilations as well as translations and its extremizers consists of $(N+2)$-dimensional set of parameters \cite{Aubin, Talenti}
\begin{align}
\mathcal{M}_{\mbox{\tiny{S}}} \coloneqq \{ c U[z, \lambda]: c \in \mathbb{R}\backslash \{0\}, \la >0, z \in\mathbb{R}^N\},
\end{align}
where 
\begin{align} \label{AT bubble}
U(x) =  \frac{C_N}{(1 + |x|^2)^{\frac{N-2}{2}}}, 
\end{align}
with normalization $ \|U\|_{L^{2^{\star}}}=1$ and $U[z,\la](x) = \la^{\frac{N-2}{2}}U(\la(x-z))$, hereafter will be called Aubin-Talenti bubbles. Note that with respect to the above normalization we have for $\gamma = 0, U_{0} = U,$ and $C_{0,N} = C_N.$ 
In the 90's Bianchi and Egnell \cite{BE-1991} improved the Sobolev inequality, by proving that the corresponding deficit functional controls the distance from the set of optimizers in the following sense. We denote, keeping the dependence on $\gamma$ implicit, the deficit functional
\begin{align}
\delta(u) \coloneqq \int_{\mathbb{R}^N} \left(|\nabla u|^2 - \frac{\gamma}{|x|^2}u^2\right) \ {\rm d}x - S_{\gamma}\|u\|_{L^{2^{\star}}}^2
\end{align}
and the distance form the set of optimizers
\begin{align*}
\mbox{dist}(u, \mathcal{M}_{\mbox{\tiny{HS}}})^2 \coloneqq \inf_{c, \lambda} \int_{\mathbb{R}^N} \left(|\nabla (u - cU_{\gamma}[\lambda])|^2 - \frac{\gamma}{|x|^2}(u - cU_{\gamma}[\lambda])^2\right) \ {\rm d}x.
\end{align*}
When $\gamma = 0,$ $\mathcal{M}_{\mbox{\tiny{HS}}}$ would have to be replaced by                                      $\mathcal{M}_{\mbox{\tiny{S}}}$
in the definition of $\mbox{dist}$ and $S_{\gamma}$ would have to be replaced by $S$ in the definition of $\delta(u).$
We know from the work of Bianchi-Egnell \cite{BE-1991} for $\gamma = 0,$ and \cite{WeiWu-2022} for $\gamma >0$ that 
\begin{align} \label{BE constant}
C_{\mbox{\tiny{BE}}}(\gamma) \coloneqq \inf \frac{\delta(u)}{\mbox{dist} (u, \mathcal{M}_{\mbox{\tiny{HS}}})^2} > 0,
\end{align}
where the infimum is taken over all $u \not \in \mathcal{M}_{\mbox{\tiny{HS}}}.$ When $\gamma = 0$ we shall denote the best Bianchi-Egnell constant by $C_{\mbox{\tiny{BE}}}$.
In this article, we are interested in the question of the existence of an optimizer attaining $C_{\mbox{\tiny{BE}}}(\gamma)$ for $\gamma >0.$ For the Sobolev inequality, it has been obtained in the interesting article by K\"onig \cite{KT-2025}. Subsequent developments has been obtained in \cite{WeiWu-2024}, for the Caffarelli-Kohn-Nirenberg inequality. As noticed first by K\"onig and in the subsequent articles, a common theme of this problem is the existence of two critical levels:
\begin{itemize}
\item[(a)] The {\it local Bianchi-Egnell critical level} $C_{\mbox{\tiny{BE}}}^{\mbox{\tiny{loc}}}.$
\item[(b)] The {\it 2-peak critical level} $C_{\mbox{\tiny{BE}}}^{\mbox{\tiny{2-peak}}}.$
\end{itemize}
Consider a minimizing sequence corresponding to $C_{\mbox{\tiny{BE}}}$. From a very crude bound $C_{\mbox{\tiny{BE}}} <1,$ it follows that a minimizing sequence is bounded in $H^1(\R^N)$. Then by classical concentration-compactness principle of Lions \cite{LP1-1984, LP2-1984}, we can assume that (up to translations and dilations) the weak limit is non-zero. In order to prove that the limit attains $C_{\mbox{\tiny{BE}}}$, one needs to show that the convergence is strong in $H^1$ and the limit is not an element of the set of optimizers. It is proved by K\"onig, that the 2-peak critical level detects the loss of strong convergence and the local Bianchi-Egnell critical levels detects the possibility of a minimizing sequence converging to one of the extremizers. Hence if $C_{\mbox{\tiny{BE}}} < \min \{C_{\mbox{\tiny{BE}}}^{\mbox{\tiny{loc}}},  C_{\mbox{\tiny{BE}}}^{\mbox{\tiny{2-peak}}} \} $ then $C_{\mbox{\tiny{BE}}}$ is attained. 
The main idea of proving $C_{\mbox{\tiny{BE}}}<C_{\mbox{\tiny{BE}}}^{\mbox{\tiny{2-peak}}}$ is quite robust in all relevant problems: test $\delta(u)/\mbox{dist} (u, \mathcal{M}_{\mbox{\tiny{HS}}})^2$ against a sum of two suitably chosen ``weakly interacting" bubbles (i.e., their $H^1$-inner product is close to 0).
However, the local Bianchi-Egnell critical level relied on a delicate property of the third eigenfunction space
of the linearized operator
\begin{align*}
-\Delta \rho = \mu_3 U^{2^{\star}-2}\rho, \ \rho \in H^1(\mathbb{R}^N).
\end{align*}
Formally, the idea is as follows, if we test the inequality \eqref{BE constant} against $U + \delta \rho$ and let $\delta \searrow 0,$ a third order expansion of the associated quantities proves that $C_{\mbox{\tiny{BE}}} < C_{\mbox{\tiny{BE}}}^{\mbox{\tiny{loc}}}$, provided the non-vanishing condition $\int_{\mathbb{R}^N} U^{2^{\star}-3}\rho^3 \neq 0$ holds.  As it seems, from our humble understanding, that the structure of the third eigenfunction is an indispensable step towards this goal. For the Sobolev inequality, the existence of a spherical harmonics of degree 2 in the third eigenspace provides the desired non-vanishing criterion. On the other hand, it has been observed in \cite{KT-exp} by K\"onig that for the one-dimensional fractional Sobolev inequality the non-vanishing criterion fails. He even went one step further to prove that $\delta(u)/\mbox{dist} (u, \mathcal{M}_{\mbox{\tiny{S}}})^2 \geq C_{\mbox{\tiny{BE}}}^{\mbox{\tiny{loc}}}$ for $u$ close to $\mathcal{M}_{\mbox{\tiny{S}}}$. This beautiful observation prompted him to conjecture that $C_{\mbox{\tiny{BE}}}=C_{\mbox{\tiny{BE}}}^{\mbox{\tiny{loc}}}$ and the best Bianchi-Egnell constant would not be achieved for the one dimensional fractional Sobolev inequality. In the same spirit, the article \cite{WeiWu-2024}, studied the attainability of the best Bianchi-Egnell constant for the CKN-inequality. It has been observed that for certain ranges of the involved parameters, the non-vanishing condition is satisfied. As a result,  a similar argument can be adopted to this case to prove the existence of a minimizer. However, for the remaining cases, it is known that the third eigenspace contains only the first spherical harmonics and hence it is inconclusive. Even after a fourth order expansion around an optimizer, the resulting expression does not yield the desired strict inequality in local Bianchi-Egnell constant. We will see in a moment that the later scenario is precisely our situation.
Added to these subtleties, extracting a non-zero weak limit from an optimizing sequence is not clear, as the Hardy-Sobolev inequality is not translation invariant. In this regard, we discovered a third critical level for the Hardy-Sobolev inequality: 
 \begin{align}
 C_{\mbox{\tiny{BE}}}(\gamma) < 1 -\frac{S_{\gamma}}{S}.
 \end{align}
 Once found it now seems quite natural but yet it is well hidden and 
 it turns out to be an important ingredient to extract a non-zero weak limit.

\medskip

\noindent
{\bf{Relation to the Caffarelli-Kohn-Nirenberg inequality.}} Consider the following special case of the CKN-inequality (see \cite{CaffarelliKohnNirenberg-1984})
\begin{align*}
S_a \left(\int_{\R^N} |x|^{-a \cdot 2^{\star}} |v(x)|^{2^{\star}} \, {\rm d}x \right)^{\frac{2}{2^{\star}}} \leq \int_{\R^N}|x|^{-2a}|\nabla v(x)|^2\, {\rm d}x,
\end{align*}
where $0 < a < \frac{N-2}{2}.$ The study of the extremals for the full range of CKN-inequality is an extensive volume of work spanning three decades. We refer to the following literatures \cite{Lieb83,CC-1993,CatrinaWang-2001,
FelliSchneider-2003,
DolbeaultEstebanTarantello-2008,
DolbeaultEstebanLossTarantello-2009,
DolbeaultEstebanLoss-2012,
DolbeaultEstebanLoss-2016}. The transformation $u(x) = |x|^{-a}v(x)$ gives (by direct computation or, see for example \cite{CatrinaWang-2001, GhoussoubRobert-2016})
\begin{align*}
\int_{\R^N}|x|^{-2a}|\nabla v(x)|^2 \,{\rm d}x = \int_{\R^N} |\nabla u(x)|^2 \,{\rm d}x - \gamma \int_{\R^N} \frac{|u(x)|^2}{|x|^2} \,{\rm d}x,
\end{align*}
where 
\begin{align*}
\gamma = a(N-2 - a).
\end{align*}
 Note that $\gamma$ is a strictly increasing function of $a,$ and as $a$ varies over $(0, \frac{N-2}{2}), \gamma$  varies over the region $(0, \frac{(N-2)^2}{4}).$ It follows that if $S_{a}$ is optimal in CKN-inequality then it is also optimal in the Hardy-Sobolev inequality, and hence $S_a = S_{\gamma}.$  The Bianchi-Egnell stability and the attainability of its extremizers has been studied in \cite{WeiWu-2022,WeiWu-2024, Wei25} for certain ranges of $a$ mentioned below. However, when restricted to the Hardy-Sobolev inequality, a close inspection at the proof reveals that extraction of a non-zero weak limit from an optimizing sequence has not been proved. As a result, the argument only establishes the attainment of radial best Bianchi-Egnell constant (i.e., the infimum is taken over radial functions). We remark that, for general CKN-inequality as considered in \cite{WeiWu-2024}, (for parameters that does not cover the HS-inequality), the extraction of a non-zero weak limit is a sub-critical problem and hence  arguments of \cite{CatrinaWang-2001} can be applied. Following their notations, we denote
\begin{align*}
a_c = \frac{N-2}{2}, \ \ \qquad \qquad \qquad a_c^{\star} = \left(1 - \sqrt{\frac{N-1}{2N}}\right)a_c.
\end{align*}

Then the Bianchi-Egnell stability holds for all $0< a <\frac{N-2}{2}$, while in \cite{WeiWu-2024} an extremizer of the best radial Bianchi-Egnell constant has been shown to exist only in the region 
\begin{align*}
  a_c^{\star}<a < a_c.
\end{align*}
 We denote, the constant
\begin{align} \label{gammastar}
\gamma_c^{\star} \coloneqq a_c^{\star}(N - 2 - a_c^{\star}) = \frac{N+1}{2N} a_c^2 = \frac{N+1}{2N}\left(\frac{N-2}{2}\right)^2.
\end{align}
Then we know the best radial Bianchi-Egnell constant $\cber(\gamma)$ is attained for $\gamma > \gamma_c^{\star}.$ 
However, we notice below that $\cber(\gamma) \equiv C_{\star},$ a constant, for all admissible $\gamma$ and consequently their result leaves an open space for a substantial improvement.
 
 \medskip
 
 We denote the discrete spectrum of the linearized operator $(-\Delta - \frac{\gamma}{|x|^2})/U_{\gamma}^{2^{\star}-2}$ by $\{\mu_k^\gamma\}_{k=1}^{\infty},$ and 
 \begin{align*}
 \Lambda(\gamma) = 1 - \frac{\mu_2^{\gamma}}{\mu_3^{\gamma}}  =: C_{\mbox{\tiny{BE}}}^{\mbox{\tiny{loc}}}(\gamma)
 \end{align*}
 the spectral gap. In \cite[Proposition~3.1]{WeiWu-2024} the authors calculated the explicit values of $\Lambda(\gamma)$ for all admissible $\gamma$ (see \eqref{lambda_gamma} below) and also showed that 
  that  $\cber(\gamma) < \Lambda(\gamma)$ for $\left[\gamma_c^{\star},\frac{(N-2)^2}{4}\right).$ It follows from the explicit expression that $\Lambda(\gamma)$ is continuous, strictly increasing for $ 0 < \gamma < \gamma_c^{\star},$ and constant for $\gamma \geq \gamma_c^{\star}.$ Whereas $\cber(\gamma)$ is constant throughout. We define

\begin{align} \label{thegamma}
\gamma_0 = 
\begin{cases}
 \mbox{the point $\Lambda(\gamma_0) = \cber(\gamma_0)$} \ \ \ \ \mbox{if} \ N \geq 4 \\
\left(1-\left(\frac{N}{N+4}\right)^{\frac{N}{N-1}}\right)\frac{(N-2)^2}{4}, \ \ \ \ \mbox{if} \ N =3.
\end{cases}
\end{align}
As a result $\cber(\gamma) < \Lambda(\gamma)$ for all $\gamma_0 < \gamma < \frac{(N-2)^2}{4}.$
Note that for $N \geq 4$, $\gamma_0 < \gamma_c^{\star},$ but for $N =3,$ $\gamma_0 > \gamma_c^{\star}.$  

\medskip

By analysing the asymptotic behaviour of $\Lambda(\gamma)$ and the $2$-peak level, the authors in \cite{Wei25} able to show that the Bianchi-Egnell extremizer is attained for a slightly larger region. We know that (see section \ref{twopeak})  the $2$-peak level lies strictly below $\Lambda(\gamma)$ for dimension $N \geq 7.$ Let $\gamma_1$ denote the point of their intersection. In \cite{Wei25}, it is shown that $\cbe(\gamma)$ is attained for $\gamma > \gamma_1$. Again we remark that, the authors studied it for the general CKN inequality, however, when restricted to the HS-inequality, the extraction of non-zero has not been shown. Since $\cber(\gamma) < 2$-peak level, we notice that $\gamma_0 < \gamma_1.$  The main result of this article is the following

\begin{Th}\label{main2} 
 Let $N \geq 3,$ and $\gamma_0$ be defined as in \eqref{thegamma}, then the best Bianchi-Egnell constant $\cbe(\gamma)$ is attained for every $\gamma_0 \leq \gamma < \frac{(N-2)^2}{4}$.  
\end{Th}

 The quantitative stability of geometric inequalities is a vast and rapidly expanding subject. It is out of the scope of this article to list all the references. For readers' interests, we tried to refer to the following small list of articles from the extensive archive. The quantitative stability of Sobolev \cite{Bl-1985,BE-1991, CFM09, ChenFrankWeth-2013, Ruffini-2014,  FigalliNeumayer-2019, Neumayer-2020, FigalliZhang-2022}, Caffarelli-Kohn-Nirenberg \cite{RuzhanskyMichaelSuraganYessirkegenov-2018, WeiWu-2022, Nitti, Zou, Lu_CKN}, on the best constant of $\cbe$ see \cite{DolbeaultEstebanFigalliFrankLoss-2025, Lu_Adv}, quantitative stability on manifolds \cite{EngelsteinNeumayerSpolaor-2022, NobiliViolo-2024, hyp25}, degenerate stability \cite{Frank_deg, FrankPeteranderl-2024, Zou_deg}, regarding constructive stability we refer to \cite{BonforteDolbeaultNazaretSimonov-2020, BDNS-2023}.
 We also refer to the expository articles \cite{Figalli_exp, Frank_exp, Dolb_exp, Dolb_exp2} for a detailed and quiet exhaustive exposition of the stability of geometric inequalities.

 The structure of the article is as follows: Section \ref{notations section} presents the notations and preliminary concepts. Section \ref{hiddenlevel} discusses the spectral gap and the hidden critical level, followed by an analysis of the $2$-peak level in Section  \ref{twopeak}. In Section \ref{proofs}, we provide the proof of Theorem \ref{main2}. Finally, Appendix (Section \ref{appendix}) includes several known results that are used throughout the article.

\section{Notations and preliminaries} \label{notations section}

We fix the notations and make a list of important constants used through out the article.
We denote
\begin{align} \label{constant1}
\var = \sqrt{\frac{(N-2)^2}{4} - \gamma}, \ \ \mbox{and} \ \ \theta = \frac{2}{N-2}\var \in (0,1)
\end{align}
where recall 
\begin{align*}
0 < \gamma < \frac{(N-2)^2}{4} \ \ \ \mbox{is the Hardy-Sobolev spectral parameter}.
\end{align*}

We also introduce the function   
\begin{align} \label{halpha}
h_{\alpha}(t)= \sech^{\frac{N-2}{2}}(\alpha t), \ \ \ \ \mbox{for} \  \alpha>0, \ t \in \R,
\end{align}
and its $L^{2^{\star}}$-scale invariant form
$$
g_{\alpha}(t) = \alpha ^{\frac{1}{2^{\star}}} h_{\alpha}(t),
$$ 
so that $\|g_{\alpha}\|_{L^{2^{\star}}(\R)}$ is independent of $\alpha$.
Another important constant relating $S_{\gamma}$ and $S$ frequently used in this article is 
$$
\ca= \left[\left(\frac{N-2}{2}\right)^2+\frac{N-2}{2}\right]\alpha^2, \ \ \ \mbox{for} \  \alpha >0.
$$
Note that by definition, $\ca = \mathcal{C}_1\alpha^2$ holds. The importance of this constant is indicated in Lemma \ref{CompareSg/S} below. This constant appears in the ODE satisfied by $h_{\alpha}$ (see Appendix, Lemma \ref{O.D.E}).

\subsection{Isometric lifting on the cylinder} \label{isometric lifting}
We refer to  \cite[Section1]{RF-2017} for detailed explanation on the topics discussed in this subsection.

\medskip

For $u \in C_c^{\infty}(\rn),$ define $\hat u \in C_c^{\infty}(\R\times\sn)$ by 
\begin{align*}
\hat{u}(t,\sigma)=\e^{-\frac{(N-2)}{2}t}u(\e^{-t}\sigma).
\end{align*}
Then 
\begin{align*}
\|\nabla u\|_{L^2(\mathbb{R}^N)}^2 = \int_{\mathbb{R} \times \mathbb{S}^{N-1}} \left(|\partial_t \hat{u}|^2 +|\nabla_{\mathbb{S}^{N-1}}\hat{u}|^2 + \frac{(N-2)^2}{4} |\hat{u}|^2 \right)\, {\rm d}t\, {\rm d}\sigma
\end{align*}
and hence $u \mapsto \hat u $ is an isometric isomorphism between the Sobolev spaces $H^1(\rn)$ and $H^1(\R \times \sn),$ where the later space is defined by the completion of $C_c^{\infty}(\R \times\sn)$ with respect to the norm on the right hand side of the above equality, and ${\rm d}\sigma$ is the unnormalized surface measure on $\sn$. The transformation preserves the $L^{2^{\star}}$-norm, while the Hardy-term changes to 
\begin{align*}
\int_{\mathbb{R}^N} \frac{|u|^2}{|x|^2} \,{\rm d}x = \int_{\R\times\sn} |\hat{u}|^2 \,{\rm d}t\,{\rm d}\sigma.
\end{align*}

Moreover, the Laplacian on $\R^N$ decomposes as 
$$
(\Delta u)(e^{-t}\sigma) = e^{\frac{N+2}{2}t}\left(\partial_{tt} + \Delta_{\sn}-\frac{(N-2)^2}{4}\right)\hat{u}(t,\sigma).
$$

The Hardy-Sobolev inequality, in cylindrical co-ordinates transforms to 
\begin{align}\label{Sobolev norms}
S_{\ga} \left(\int_{\mathbb{R}\times \mathbb{S}^{N-1}}|\hat{u}|^{2^{\star}} \,{\rm d}t\, {\rm d}\sigma\right)^{\frac{2}{2^{\star}}} \leq \int_{\mathbb{R}\times \mathbb{S}^{N-1}} \left( |\partial_t\hat{u}|^2 + |\nabla_{\mathbb{S}^{N-1}}\hat{u}|^2 + \var^2|\hat{u}|^2\right)\, {\rm d}t\,{\rm d}\sigma,
\end{align}
for all $\hat{u} \in H^1(\R \times \sn)$. For simplicity of notations, we denote the norm on the right hand side of \eqref{Sobolev norms} by $\|\hat{u}\|_{\var}^2,$ and the inner-product inducing this norm by $\langle \cdot, \cdot \rangle_{\var}.$ As a result, 
\begin{align*}
\langle \hat{u}, \hat{v} \rangle_{\var} = \int_{\R^N} \left(\nabla u \cdot \nabla v - \frac{\gamma}{|x|^2}uv\right) \,{\rm d}x =: \langle u, v \rangle_{\gamma}
\end{align*}
holds. When $\var=\frac{N-2}{2}$ (i.e., when $\gamma = 0$), we symbolize the norm by $\|\cdot\|_{H^1(\R\times\sn)}$ and the inner product by $\langle\cdot ,\cdot\rangle_{H^1(\R\times\sn)}$. Note that if $u \in H^1_{\mbox{\tiny{rad}}}(\R^N)$ then $\hat{u}$ is independent of $\sigma$ variable and both $\hat{u}, \hat{u}^{\prime}$ are square integrable. We designate this by saying $\hat{u} \in H^1(\R).$ For this reason, we will denote 
$\widehat{H^1_{\mbox{\tiny{rad}}}}(\R^N)$ by $H^1(\R).$ 

\medskip

\medskip

\subsection{The optimizers in cylindrical co-ordinate}
Since both $U_{\gamma},\, U$ are radial, under the above transformation, both $\hat U_{\gamma}$ and $\hat U$ are independent of the $\sigma$-variable. In addition,
\begin{itemize}
\item[$\bullet$] The Hardy-Sobolev optimizer $\ug$ transforms to 

\begin{align*}
\hug(t) \coloneqq \e^{-\frac{N-2}{2}t}\ug(e^{-t}\sigma) 
=2^{-\frac{N-2}{2}}\cgn\sech^{\frac{N-2}{2}}(\theta t) \ \mbox{where recall} \ \theta=\frac{2}{N-2}\var.
\end{align*}
\item[$\bullet$] The Aubin-Talenti bubble $\hat U$ corresponds to $\theta = 1$ (or $\gamma = 0$),
\begin{align*}
\hat U(t) = 2^{-\frac{N-2}{2}}C_N\sech^{\frac{N-2}{2}}(t)
\end{align*}
where we recall from the introduction that $C_{0,N} = C_N$ (defined in \eqref{HS bubble} and \eqref{AT bubble})
and both respects the normalizations $\|\hug\|_{L^{2^{\star}}(\R\times\sn)}=1=\|\hat{U}\|_{L^{2^{\star}}(\R\times\sn)}$.
\end{itemize}

\medskip
From \eqref{halpha} and \eqref{constant1} we see that $\hat{U}_{\gamma}$ is a constant multiple of $h_{\theta}$ (and hence $g_{\theta}$) and $\hat{U}$ is a constant multiple of $h_1$(and hence $g_1$).

\medskip

$\bullet$ The dilation parameter $\la$  now transforms to the translation parameter $s = -\ln \la$ in this new co-ordinates.  We denote the set of Hardy-Sobolev extremizers in cylindrical co-ordinates by 
\be\label{HS-Ext-Cyl}
\hat{\mathcal{M}}_{\mbox{\tiny{HS}}} \coloneqq \left\{c\hat{U}_{\gamma}[s] \ : \ (c,s)\in\R\times\R,\, \|\hat{U}_{\gamma}[s]\|_{L^{2^{\star}}(\R\times\sn)}=1\right\},
\ee
where $\hat{U}_{\gamma}[s](\cdot) \coloneqq \hat{U}_{\gamma}(\cdot + s),\, s \in \R.$ In addition, 
\begin{eqnarray*}
	\di(u,\mhs)^2 = \inf_{(c,\lambda)\in \R\times \R^{+}}  \left\|u-cU_{\gamma}[\lambda]\right\|_{\gamma}^2 = 
	 \inf_{(c,s)\in \R\times \R} \left\|\hat{u}-c\hat{U}_{\gamma}[s]\right\|_{\var}^2
 =: \di (\hat{u},\hat{\mathcal{M}}_{\mbox{\tiny{HS}}})^2,
\end{eqnarray*}
holds.

\medskip

\medskip

Next we prove the following well known lemma (see for example, \cite{DolbeaultEstebanLossTarantello-2009, GhoussoubRobert-2016}, Beckner,W. (arXiv:0907.3932v1)) comparing the values $\sg$ and $S$. 

\begin{Lem}\label{CompareSg/S}
$S_{\gamma}= \theta^{1+\frac{2}{2^{\star}}}S$. 
\end{Lem}
\begin{proof}
We have seen that constant multiple of $h_1$ is the Sobolev minimizer whereas a constant multiple of $h_{\theta}$ is a Hardy-Sobolev minimizer on the cylinder. Therefore, by homogeneity
\begin{align*}
\frac{S}{|\sn|^{1-\frac{2}{2^{\star}}}}=\frac{\intR |h_{1}'|^2+\frac{(N-2)^2}{4} |h_{1}|^2}{(\intR h_{1}^{2^{\star}})^{\frac{2}{2^{\star}}}}
=\mathcal{C}_{1}\|h_{1}\|_{\lts(\R)}^{2^{\star}(1-\frac{2}{2^{\star}})}
\end{align*}
 whereas, by Lemma \ref{O.D.E} and \eqref{constant1}
\begin{align*}
\frac{\sg}{|\sn|^{1-\frac{2}{2^{\star}}}}&=\frac{\intR |h_{\theta}'|^2+\var^2 |h_{\theta}|^2}{(\intR h_{\theta}^{2^{\star}})^{\frac{2}{2^{\star}}}}
=\mathcal{C}_{\theta}\|h_{\theta}\|_{\lts(\R)}^{2^{\star}(1-\frac{2}{2^{\star}})}
=\frac{\mathcal{C}_{\theta}}{\theta^{1-\frac{2}{2^{\star}}}}\|h_{1}\|_{\lts(\R)}^{2^{\star}(1-\frac{2}{2^{\star}})}
\end{align*}
where $|\sn|$ denotes the surface area of $\sn.$ Using the relation $\mathcal{C}_{\theta}=\mathcal{C}_1\theta^2$, we conclude the proof. 
\end{proof}

In order to effectively manipulate the distance from the set of optimizers, as in \cite{KT-2025} we define 
\be\nonumber
  \m (u) \coloneqq \sup_{\la >0}\left(\intr \ug[\la]^{2^{\star}-1}u\ {\rm d}x\right)^2.
\ee
Then we can reformulate $\mbox{dist}(u,\mathcal{M}_{\mbox{\tiny HS}})$ in terms of $\mathfrak{m}(u)$.

\begin{Lem}\label{dofm}
	For every $u \in H^1(\R^{N})$
	\be\label{ditm}
	\mbox{dist}(u,\mathcal{M}_{\mbox{\tiny HS}})^2 = \|u\|_{\gamma}^2-S_{\ga}\mathfrak{m}(u)	\ee
 and $\mbox{dist}(u,\mathcal{M}_{\mbox{\tiny HS}})$ is achieved. 
 
 \medskip
 
   In addition, if $\mathfrak{m}(u)>0$, then  the function $\left(\int_{\rn}uU_{\gamma}[\lambda]^{2^{\star}-1}\,{\rm d}x\right)U_{\gamma}[\lambda]$ optimizes $\di(u,\mathcal{M}_{\mbox{\tiny HS}})$ if and only if $U_{\gamma}[\lambda]\in \mathcal{M}_{\mbox{\tiny HS}}$ optimizes $\mathfrak{m}(u)$.
\end{Lem}

\begin{proof}
	
Using the equation satisfied by $U_{\ga}[\la],$ we complete the square
	\begin{eqnarray*}
&& \ \mbox{dist} (u, \mathcal{M}_{\mbox{\tiny{HS}}})^2 = \inf_{(c,\la)\in \R\times\R^{+}} \left\|u-cU_{\ga}[\la]\right\|_{\gamma}^2 \\
	&=&\inf_{\lambda>0}\inf_{c\in\R}\left[ \|u\|_{\gamma}^2 -S_{\gamma}\left(\int_{\R^{N}}uU_{\gamma}[\lambda]^{2^{\star}-1}\,{\rm d}x\right)^2+S_{\gamma}\left(c-\int_{\R^{N}}uU_{\gamma}[\lambda]^{2^{\star}-1}\,{\rm d}x\right)^2\right]\\
	&=&\|u\|_{\gamma}^2-S_{\gamma}\sup_{\lambda>0}\left(\int_{\R^{N}}uU_{\gamma}[\lambda]^{2^{\star}-1}\,{\rm d}x\right)^2,
	\end{eqnarray*}
	as desired. From the above identity it follows that if $\m (u)$ is achieved at $U_{\ga}[\la_0],$ then $\mbox{dist}(u, \mathcal{M}_{\mbox{\tiny HS}})$ is achieved at $c_0U_{\ga}[\la_0],$ where 
 $c_0=\int_{\R^{N}}uU_{\gamma}[\lambda_0]^{2^{\star}-1}\,{\rm d}x$, and vice versa.
	
	\medskip
	
	\noindent In order to complete the proof it remains to show that $\mathfrak{m}(u)$ is always achieved. Notice that, there are two cases:

 \medskip
 
\noindent {\bf{Case 1:}}	 
 $\mathfrak{m}(u)=0$, then 
$\mathfrak{m}(u)=\left(\int_{\R^{N}}uU_{\gamma}[\lambda]^{2^{\star}-1}\,{\rm d}x\right)^2 = 0$ for all $\lambda>0$	and $\di(u,\mhs)$ is achieved at $0\in\mhs$.

 \medskip
 
\noindent{\bf{Case 2:}} If $\mathfrak{m}(u)>0$, then we can find a sequence $\{\la_n\}\subset \R^{+}$ such that 
	\be\nonumber
	\left(\intr uU_{\gamma}[\lambda_n]^{2^*-1}\,{\rm d}x\right)^2\longrightarrow \mathfrak{m}(u)>0, \text{ as }n\to \infty.
	\ee
	Now if $\left|\ln \lambda_n\right|\to\infty$ as $n\to\infty$, then $U_{\ga}[\la_n]\rightharpoonup 0$ in $H^1(\rn)$ and hence $\mathfrak{m}(u)=0$. So, $\mathfrak{m}(u)>0$ implies $\lambda_n\to\lambda_0>0$ and $\mathfrak{m}(u)$ is achieved at $U_{\ga}[\la_0].$
\end{proof}

\section{The Spectral gap and hidden critical level}\label{hiddenlevel}

We recall the radial Bianchi-Egnell best constant is given by
\begin{align}\label{cberadial}
\cber(\gamma) \coloneqq \inf _{\substack{u \not\in \mhs,\\  u \in H^1_{\tiny\mbox{rad}}(\rn)}}\frac{ \|u\|_{\gamma}^2- \sg\|u\|_{\lts(\R^N)}^2}{\di(u,\mhs)^2}.
\end{align}
Clearly  $\cbe(\gamma)\leq \cber(\gamma) $. The next proposition indicates its purpose.
\begin{Prop}\label{constantcber}
 $\cber(\gamma)$ is constant for all $\gamma\in(0, \frac{(N-2)^2}{4})$ and is attained.
\end{Prop}

 It is convenient to work in the cylindrical co-ordinates. We denote by $\hat{\m}(\hat u)$ the function introduced above, and to keep track of $\bf\theta$ we emphasize the notation $\hat{\m}_{\bf\theta}(\hat u)$ and $\hat{\mathcal{M}}_{\gamma}$ the set of extremizers defined in \eqref{HS-Ext-Cyl}.
Therefore,
\begin{align*}
\hat{\m}_{\bf\theta}(\hat{u}) \coloneqq \sup_{s\in\R}\left(\int_{\R\times \sn} \hat{U}_{\gamma}[s]^{2^\star-1}\hat{u}(t,\sigma)\,{\rm d}t \,\ds\right)^2.
\end{align*}

The proof of Proposition \ref{constantcber} follows from scaling properties of $\hat{m}_{\theta},\|\cdot\|_{\varepsilon}$ and $\|\cdot\|_{\lts}$ mentioned in the following lemma.

\begin{Lem}
Let $\theta_1, \theta_2\in(0,1), \varepsilon_i=\frac{N-2}{2}\theta_i,$ and  $\hat{u}\in H^1(\R\times \sn)$ then,
\begin{itemize}
\item[(a)] $\|\hat u\|_{2^\star}^2=\left(\frac{\theta_2}{\theta_1}\right)^{\frac{2}{2^{\star}}}\left(\int_{\R\times\sn} \hat{u}^{2^{\star}}(\frac{\theta_2}{\theta_1}t,\sigma){\rm d}t\,\ds\right)^{\frac{2}{2^{\star}}}$.

\item[(b)] $\hat{\m}_{\theta_1}(\hat{u})=\left(\frac{\theta_2}{\theta_1}\right)^{\frac{2}{2^{\star}}}\hat{m}_{\theta_2}(\hat{u}(\frac{\theta_2}{\theta_1}\cdot, \cdot)).$
\item[(c)] If $\hat{u}$ is independent of $\sigma$-variable, then
$\|\hat u\|_{\varepsilon_1}^2=\frac{\theta_1}{\theta_2}\left\|\hat u\left(\frac{\theta_2}{\theta_1}\cdot\right)\right\|_{\varepsilon_2}^2.$
\end{itemize}
\end{Lem}
\begin{proof}
(a) follows form the change of variable $t\leftrightarrow\frac{\theta_2}{\theta_1}t.$ 

\noindent
(b) By Lemma \ref{dofm} there exist $s\in \R$ such that $\hat{\m}_{\theta_1}(\hat{u})$ is attained. Changing the variable $t\leftrightarrow\frac{\theta_2}{\theta_1}t$ we get
\begin{align*}
\hat{\m}_{\theta_1}(\hat{u})&=\left(\frac{\theta_2}{\theta_1}\right)^2\left(\int_{\R \times \sn}\hat{u} \left(\frac{\theta_2}{\theta_1}t, \sigma \right)\frac{\sech^{N+2}(\theta_2t +\theta_1 s)}{\|\sech(\theta_1\cdot)\|_{\lts}^{2^{\star}-1}}{\rm d}t\,\ds\right)^2\\
&=\left(\frac{\theta_2}{\theta_1}\right)^2\left(\frac{\theta_1}{\theta_2}\right)^{\frac{2(2^{\star}-1)}{2^{\star}}}\left(\int_{\R \times \sn} \hat{u} \left(\frac{\theta_2}{\theta_1}t, \sigma \right)\frac{\sech^{N+2}(\theta_2t +\theta_1 s)}{\|\sech(\theta_2\cdot)\|_{\lts}^{2^{\star}-1}}{\rm d}t\,\ds\right)^2\\
&\leq \left(\frac{\theta_2}{\theta_1}\right)^{\frac{2}{2^{\star}}}\hat{\m}_{\theta_2}\left(\hat{u}\left(\frac{\theta_2}{\theta_1}\cdot, \cdot\right)\right).
\end{align*}
Interchanging the role of $\theta_1,\theta_2$, we get
\begin{align*}
\hat{\m}_{\theta_2}(\hat{u})\leq \left(\frac{\theta_1}{\theta_2}\right)^{\frac{2}{2^{\star}}}\hat{\m}_{\theta_1}\left(\hat{u}\left(\frac{\theta_1}{\theta_2}\cdot, \cdot\right)\right)
\end{align*}
and hence equality holds.

\noindent
(c)  We do the similar change of variable as in part(a) with $t=\frac{\theta_2}{\theta_1}s$
\begin{align*}
\|\hat{u}\|_{\varepsilon_1}^2&=\int_{\R\times\sn}\left(\left|\partial_t \hat{u}\right|^2+\varepsilon_1^2 \hat{u}^2\right)\,{\rm d}t\,\ds\\
&=\left(\frac{\theta_1}{\theta_2}\right)\int_{\R\times\sn}\left(\left|\partial_s\left(\hat{u}\left(\frac{\theta_2}{\theta_1}s\right)\right)\right|^2 +\varepsilon_2^2 \hat{u}^2\left(\frac{\theta_2}{\theta_1}s\right)\right)\,{\rm d}s\,\ds\\
&=\left(\frac{\theta_1}{\theta_2}\right)\left\|\hat{u}\left(\frac{\theta_2}{\theta_1}\cdot\right)\right\|_{\varepsilon_2}^2,
\end{align*}
where in the second step we used $\varepsilon_1=\frac{\theta_1}{\theta_2}\varepsilon_2.$
This completes the proof of the Lemma.
\end{proof}

\medskip

\noindent 
 {\bf{Proof of Proposition \ref{constantcber} }}
 \begin{proof}
Given admissible $\gamma_1, \gamma_2,$ we associate $\theta_1, \theta_2$ by the relation \eqref{constant1}. By Lemma \ref{CompareSg/S}, $S_{\gamma}=\theta^{(1+\frac{2}{2^{\star}})}S$ and hence \begin{align*}
 \left(\frac{\theta_2}{\theta_1}\right)^{1+\frac{2}{2^{\star}}}S_{\gamma_1}=S_{\gamma_2}.
\end{align*} 
Moreover, $\hat\phi \in \hat{\mathcal{M}}_{\gamma_1}$ if and only if $\hat\phi \left(\frac{\theta_2}{\theta_1} \cdot\right) \in \hat{\mathcal{M}}_{\gamma_2}.$ Also, recall our convention $\widehat{H^1_{\mbox{\tiny{rad}}}}(\R^N)= H^1(\R).$ 
For $\hat u \in H^1(\R),$ we denote $\hat v(t) = u\left(\frac{\theta_2}{\theta_1}t\right)$. Then,
\begin{align*}
\cber(\gamma_1)&=\inf_{\substack{\hat{u}\in H^1(\R)\setminus\{0\} \\  \hat{u}\not\in\hat{\mathcal{M}}_{\gamma_1}}}\frac{\|\hat u\|_{\varepsilon_1}^2-S_{\gamma_1}\|\hat{u}\|_{\lts}^2}{\|\hat{u}\|_{\varepsilon_1}^2-S_{\gamma_1}\hat{\m}_{\theta_1}(\hat{u})}\\
&=\inf_{\substack{\hat v\in H^1(\R)\setminus\{0\} \\  \hat v\not\in\hat{\mathcal{M}}_{\gamma_2}}}\frac{\frac{\theta_1}{\theta_2}\|\hat v\|_{\varepsilon_2}^2-S_{\gamma_1}\left(\frac{\theta_2}{\theta_1}\right)^{\frac{2}{2^{\star}}}\|\hat v\|_{\lts}^2}{\frac{\theta_1}{\theta_2}\|\hat v\|_{\varepsilon_2}^2-S_{\gamma_1}\left(\frac{\theta_2}{\theta_1}\right)^{\frac{2}{2^{\star}}}\hat{\m}_{\theta_2}(\hat v)}\\
&=\cber(\gamma_2).
\end{align*} 

If $\gamma_1 \geq \gamma_c^{\star},$ then we know there exists a minimizer, denote it by $\hat\phi_{\gamma_1}.$ Given any $\gamma_2,$ we scale it as above to obtain the desired minimizer.
 \end{proof}

\subsection{$\cbe(\gamma)$ less than the hidden critical level}
We prove that $\cbe(\gamma)$ is below the hidden critical level $1 - \frac{S_{\gamma}}{S}$ for every $\gamma\in \left(0,\frac{(N-2)^2}{4}\right)$ and $N \geq 4,$ an essential ingredient to obtain a non-zero weak limit.

\medskip

\subsubsection{Spectral Gap Inequality}
Wei-Wu in \cite[Proposition~3.1]{WeiWu-2024} obtained the precise spectral gap in context of CKN-inequality and hence the local Bianchi-Egnell critical level $\cbel(\gamma) \coloneqq 1-\frac{\mg_2}{\mg_3}$ in the full range $0 < \gamma <  \frac{(N-2)^2}{4}$ (in their article for $0<a< (N-2)/2$) and it is given by
\begin{align} \label{lambda_gamma}
\Lambda(\gamma)=
\begin{cases}
\frac{2+2q(a)+(2^{\star}-2)(1+q(a))^{\frac{1}{2}}-2^{\star}(2^{\star}-1)}{2+2q(a)+(2^{\star}-2)(1+q(a))^{\frac{1}{2}}},\ \ \ & \ \mbox{if} \ 0< \gamma\leq \gamma_c^{\star},  \\
\frac{4}{N+4},\ \ \ & \mbox{if} \ \gamma_c^{\star}< \gamma< \frac{(N-2)^2}{4},
\end{cases}
\end{align}
where $q(a)=\frac{N-1}{\var^2}$. Recall that $\gamma$ and $a$ are related by $\gamma = a (N-2-a)$ and $\gamma_c^{\star}$ is defined in \eqref{gammastar}. Also see Oliver Rey \cite[Appendix~D]{RO-1990} for the spectral gap in the case of Sobolev Inequality.  The next lemma shows that in this region, the spectral gap is indeed strictly below the hidden critical threshold. We  remark that by testing $\delta(u)/\mbox{dist} (u, \mathcal{M}_{\mbox{\tiny{HS}}})^2$ against an Aubin-Talenti bubble $U[z,1]$ and letting $|z| \rightarrow \infty$ provides $\cbe(\gamma) \leq 1 - \frac{S_{\gamma}}{S}.$ However, we also need $\cbe(\gamma) < \cbel(\gamma).$ Therefore, it is convenient to compare $\Lambda(\gamma)$ and $1 - \frac{S_{\gamma}}{S}.$
\begin{Lem}\label{cbel less hidden level}
We have $\Lambda(\gamma)< 1-\frac{\sg}{S},$ for 
\begin{itemize}
\item[(a)] either $N \geq 4$ and $0 < \gamma < \frac{(N-2)^2}{4},$ 

\item[(b)] or,  $N=3$ and $ \left(1-\left(\frac{N}{N+4}\right)^{\frac{N}{N-1}}\right)\frac{(N-2)^2}{4}<\gamma < \frac{(N-2)^2}{4}$.
\end{itemize}
\end{Lem}
\begin{proof}
Recall that $\gamma$ and $\theta$ are related by $\theta^2 = 1 - \frac{4\gamma}{(N-2)^2}.$
By Lemma \ref{CompareSg/S}, $\frac{S_{\gamma}}{S} = \theta^{1 + \frac{2}{2^{\star}}}.$ A simplification of $\Lambda(\gamma)$ in the region $0<\gamma\leq \gamma_c^{\star}$ gives
\begin{align*}
1 - \Lambda(\gamma) &= \frac{ 2^{\star} (2^{\star} - 1)}{2+2q(a)+(2^{\star}-2)(1+q(a))^{\frac{1}{2}}} \\
&=  \frac{\frac{2N(N+2)}{(N-2)^2}}{2+2q(a)+(2^{\star}-2)(1+q(a))^{\frac{1}{2}}}
\end{align*}
On the other hand, using the relation $\theta^2=\frac{4\var^2}{(N-2)^2}$ the denominator can be simplified to 
\begin{align*}
2+2q(a)+(2^{\star}-2)(1+q(a))^{\frac{1}{2}} &= 2 \left(1+\frac{N-1}{\var^2}+\frac{2}{N-2}\sqrt{1+\frac{(N-1)}{\var^2}}\right)\\
& = 2\left( 1 + \frac{4(N-1)}{(N-2)^2 \theta^2}+\frac{2}{N-2}\sqrt{1+\frac{4(N-1)}{(N-2)^2 \theta^2}}  \ \right).
\end{align*}
Therefore, the desired inequality in the region $0<\gamma\leq \gamma_c^{\star}$ is equivalent to  $\theta^{1 + \frac{2}{2^{\star}}} < 1 - \Lambda(\gamma),$ i.e.,
\begin{align*}
\theta^{1 + \frac{2}{2^{\star}}}\left( 1 + \frac{4(N-1)}{(N-2)^2 \theta^2}+\frac{2}{N-2}\sqrt{1+\frac{4(N-1)}{(N-2)^2 \theta^2}}\right) < \frac{N(N+2)}{(N-2)^2}
\end{align*}
i.e.,
\begin{align} \label{Suffcond1}
\theta^{-\frac{2}{N}}\left( \theta^2 + \frac{4(N-1)}{(N-2)^2}+\frac{2\theta}{N-2}\sqrt{\theta^2+\frac{4(N-1)}{(N-2)^2 }}\right) < \frac{N(N+2)}{(N-2)^2}.
\end{align}
Whereas, in the region $\gamma_c^{\star}<\gamma<\frac{(N-2)^2}{4}$, the inequality $\theta^{1+\frac{2}{2^{\star}}}<1-\Lambda(\gamma)$, transforms to the requirement of 
\begin{align}\label{Suffcond2}
\theta^{\frac{2(N-1)}{N}}< \frac{N}{N+4}.
\end{align}
We define the functions
\begin{align}
f_1(\theta)&\coloneqq\frac{N}{N+4}-\theta^{\frac{2(N-1)}{N}}, \label{f1}\\
f_2(\theta)&\coloneqq\frac{N(N+2)}{(N-2)^2}\theta^{\frac{2}{N}}-\left(\theta^2+\frac{4(N-1)}{(N-2)^2}+\frac{2\theta}{N-2}\sqrt{\theta^2+\frac{4(N-1)}{(N-2)^2}}\right) \label{f2}.
\end{align}
Therefore,
to prove the sufficient conditions \eqref{Suffcond1} and \eqref{Suffcond2} we have to show that
  $f_1(\theta)>0$  for $\theta\in \left(0,\sqrt{\frac{N-1}{2N}}\right)$ and $f_2(\theta)>0$ for $\theta\in \left[\sqrt{\frac{N-1}{2N}},1\right)$ .
By Lemma \ref{positive} below, this is true for $N \geq 3$, except when $N=3$, \eqref{Suffcond2} is true only in the region $\theta\in \left(0,\left(\frac{N}{N+4}\right)^{\frac{N}{2(N-1)}}\right)$. This completes the proof. 

\end{proof}
\begin{Lem} \label{positive}
Let $f_1$ and $f_2$ be as in \eqref{f1} and \eqref{f2} respectively.
 \begin{itemize}
\item[(a)] If $N\geq 4$ then
\begin{align*}
f_1(\theta)>0 \ \mbox{for all} \ \theta\in \left(0,\sqrt{\frac{N-1}{2N}}\right), \  \ \mbox{and} \ \ 
f_2(\theta)>0 \ \mbox{for all} \ \theta\in \left[\sqrt{\frac{N-1}{2N}},1\right).
\end{align*}
\item[(b)] If $N=3,$ then $f_1(\theta)>0$ for $0 < \theta < \left(\frac{N}{N+4}\right)^{\frac{N}{2(N-1)}}$. 
\end{itemize}
\begin{proof} First we consider $f_1$. Since it is decreasing, enough to show
\begin{align*}
f_1\left(\sqrt{\frac{N-1}{2N}}\right)=\frac{N}{N+4}-\left(\frac{N-1}{2N}\right)^{\frac{N-1}{N}}>0
\end{align*}
Let $x=1-\frac{1}{N} \in (0,1).$ Then the required inequality is $\frac{1}{5-4x}>\left(\frac{x}{2}\right)^x$ or equivalently, 
\begin{align*}
\ln(5-4x)+x\ln x+\ln \left(\frac{1}{2}\right)x<0.
\end{align*}
Using $\ln (5-4x)<4(1-x)$ and $x\ln x<0$, the above inequality is true for $4-(4+\ln 2)x<0$, that is 
\begin{align*}
N>\frac{4+\ln 2}{\ln 2}\approx 6.77.
\end{align*}   
For $N=4,5,6$, we plug  in directly and compute $f_1\left(\sqrt{\frac{N-1}{2N}}\right)$ to show it is positive and for $N= 3$, $f_1(\theta)>0$ for $ \theta<\left(\frac{N}{N+4}\right)^{\frac{N}{2(N-1)}}$.

\medskip

We now consider $f_2.$ Note that  $f$ is strictly concave and hence has at-most two zeros. 
 Moreover $f_2(\theta)\to -\infty$ whenever $\theta\to \pm \infty$.
 Since
\begin{align*}
f_2(1)=\frac{N(N+2)}{(N-2)^2}-\left(1+\frac{4(N-1)}{(N-2)^2}+\frac{2}{N-2}\sqrt{1+\frac{4(N-1)}{(N-2)^2}}\right)=0.
\end{align*}
and,
\begin{align*}
&\quad\quad\quad\ f_2\left(\sqrt{\frac{N-1}{2N}}\right)\\
&=\frac{N(N+2)}{(N-2)^2}\left(\frac{N-1}{2N}\right)^\frac{1}{N}-\left(\frac{N-1}{2N}+\frac{4(N-1)}{(N-2)^2}+\frac{2}{(N-2)}\frac{N-1}{2N}\sqrt{1+\frac{8N(N-1)}{(N-2)^2(N-1)}}\right)\\
&=\frac{(N+2)(N+4)}{(N-2)^2}\left(\frac{N-1}{2N}\right)^{\frac{1}{N}}\left\{\frac{N}{N+4}-\left(\frac{N-1}{2N}\right)^{1-\frac{1}{N}}\right\}>0,
\end{align*}
by part (a), we conclude  $f_2(\theta)>0$ for all $\theta\in\left(\sqrt{\frac{N-1}{2N}},1\right)$ and $N\geq 4$.  
This completes the proof of the lemma.
\end{proof} 
\end{Lem}

\section{ The Bianchi-Egnell $2$-peak critical levels}\label{twopeak}

 In this section we prove $C_{\mbox{\tiny BE}}^{\mbox{\tiny rad}}(\gamma) < C_{\mbox{\tiny BE}}^{\mbox{\tiny{2-peak}}}(\gamma).$ Due to the relation between the Hardy-Sobolev and the CKN inequality mentioned in the introduction, one can directly use the results of \cite{WeiWu-2024} to conclude the $2$-peak level. However, we thought to include it because of the following reasons.  Since \cite{WeiWu-2024} treated the general CKN inequality singling out the exact result for the Hardy-Sobolev is cumbersome and in their argument they implicitly used Kelvin transform, as in \cite{KT-2025}, to obtain their results. On the other hand, our argument is (primarily) independent of the use of the Kelvin transform (though lifting to the cylindrical co-ordinates involve Kelvin transform) and we do not need the explicit structure of the optimizers. Our arguments are general enough so that it can be applied to other context where only the asymptotic behaviour of the extremizers are known.

\begin{Lem}\label{2-peak-th}
	For $0<\gamma< \frac{(N-2)^2}{4}$, we have $\cber(\gamma) <2-2^{\frac{2}{2^{\star}}}$.
\end{Lem}

Following the preceding articles, we test 
$\frac{\delta(u)}{\mbox{dist} (u, \ \mathcal{M}_{\mbox{\tiny{HS}}})^2}$ against a family of test functions of the form
\be\nonumber
v_{\lambda} \coloneqq U_{\gamma} + U_{\gamma}[\lambda],\text{ as }\lambda \to 0.
\ee  The essential ideas are in spirit of K\"onig \cite{KT-2025}. For the sake of completeness, we give the proof of these results adapting the ideas of K\"onig in our context. Once again we recall the explicit expression of the $L^{2^{\star}}$-normalized HS-extremizer
\begin{align*}
\hat{U}_{\gamma}[s](t,\sigma)= \hat{U}_{\gamma}[s](t)=\frac{\sech^{\frac{N-2}{2}} (\theta(t+s))}{\int_{\R\times\mathbb{S}^{N-1}}\sech^N (\theta t)\,{\rm d}t\,{\rm d}\sigma}
\end{align*}
where ${\theta}=\frac{2\var}{N-2} = \sqrt{1 - \frac{4\gamma}{(N-2)^2}}$, as it would be useful shortly. Since under the transformation $u \mapsto \hat{u},$ all the quantities transform suitably, we have
\begin{align*}
	\di (\hat{u},\hat{\mathcal{M}}_{\mbox{\tiny{HS}}})^2 
	= \|\hat{u}\|_{\var}^2 - S_{\gamma} \hat{\mathfrak{m}}(\hat{u}),
\end{align*}
where $ \hat{\mathfrak{m}}(\hat{u}) \coloneqq \sup_{s\in\R}\left(\int_{\R\times \mathbb{S}^{N-1}}\hat{u}\hat{U}_{\gamma}[s]^{2^{\star}-1}\,{\rm d}t\,{\rm d}\sigma\right)^2.$

The following proposition provides the necessary expansion needed to prove the $2$-peak level.  We denote the weak interaction between two bubbles by
$$
Q(s)= \int_{\R\times \mathbb{S}^{N-1}}\hat{U}_{\gamma}\hat{U}_{\gamma}[s]^{2^{\star}-1}\,{\rm d}t\,{\rm d}\sigma=\int_{\R\times \mathbb{S}^{N-1}}\hat{U}_{\gamma}[s]\hat{U}_{\gamma}^{2^{\star}-1}\,{\rm d}t\,{\rm d}\sigma = \langle \hat{U}_{\gamma}, \hat{U}_{\gamma}[s] \rangle_{\var}.
$$
Then $Q(s) \approx e^{-s \var}$ as $s \rightarrow \infty$ (see Remark \ref{bub-int-cyl} below).
\begin{Prop}\label{2-peak-key-est}
Let  $\hat{v}_s = \hat{U}_{\gamma}+\hat{U}_{\gamma}[s]$. Then as $s\to \infty$ we have the following results: 
\begin{itemize}
\item[(a)] $\|\hat{v}_{s}\|_{\varepsilon}^2 = 2S_{\gamma} + 2S_{\gamma}Q(s).$
\item[(b)]  $\|\hat{v}_{s}\|_{L^{2^{\star}}(\R\times\sn)}^2= 2^{\frac{2}{2^{\star}}} + 2^{\frac{2}{2^{\star}}+1} Q(s) + o\left(Q(s)\right)$.
\item[(c)] $\di \left(\hat{v}_{s},\hat{\mathcal{M}}_{\mbox{\tiny{HS}}}\right)^2 = S_{\gamma} + o\left(Q(s)\right).$
\end{itemize}
\end{Prop}

\begin{proof}
	
For $(a)$ we simply expand and use the equation satisfies by $\hat{U}_{\gamma},\,\hat{U}_{\gamma}[s]$
\begin{align*}
		\|\hat{v}_{s}\|_{\varepsilon}^2 = \|\hat{U}_{\gamma}\|_{\varepsilon}^2 + \|\hat{U}_{\gamma}[s]\|_{\varepsilon}^2 + 2\langle \hat{U}_{s}, \hat{U}_{\gamma}[s]\rangle_{\varepsilon} = 2S_{\gamma} + 2S_{\gamma}Q(s).
	\end{align*}
	
\noindent
Proof of (b): We decompose $\R^N$  into two disjoint regions $\mathcal{A}=\left\{t\in\R\,:\, \hug\leq\hug[s]\right\}$ and $\mathcal{A}^c=\left\{t\in\R\,:\,\hug>\hug[s]\right\}$.
	Now we expand $\left(\hug+\hug[s]\right)^{2^{\star}}$ in $\mathcal{A}$, using Taylor's series expansion 
	\begin{align*}
		\left(\hug+\hug[s]\right)^{2^{\star}}&=\hug[s]^{2^{\star}}\left(1+\frac{\hug}{\hug[s]}\right)^{2^{\star}}\\
		&=\hug[s]^{2^{\star}}+2^{\star}\hug[s]^{2^{\star}-1}\hug+O\left(\hug[s]^{2^{\star}-2}\hug^2\right).
	\end{align*}
	We now claim the following
	\begin{align}\label{bub-est-3}
		\int_{\mathcal{A}^c} \hug[s]^{2^\star-1}\hug\,{\rm d}t\,\ds=o(Q(s))=\int_{\mathcal{A}^c} \hug[s]^{2^\star}\,{\rm d}t\,\ds.
	\end{align}
	
	Indeed using the inequality $\hug>\hug[s]$ on $\mathcal{A}^c$ and adjusting the exponents as $\left(2^{\star}-1-\delta\right)$ for $\hat{U}_{\gamma}[s]$ and $(1+\delta)$ for $\hat{U}_{\gamma}$ with $0<\delta<2^{\star}-2$ and using the interaction estimate (Remark~\ref{bub-int-cyl}) we get the first equality in \eqref{bub-est-3}. For the second equality in  \eqref{bub-est-3}, we adjust the exponents $2^{\star}$ as $2^{\star}-\delta$ and $\delta$  such that $\min\{\delta,2^{\star}-\delta\}>1$.
	 Hence we have the following estimates:
	\begin{eqnarray*}
		&&\int_{\mathcal{A}} \hug[s]^{2^\star}{\rm d}t\,\ds=1-\int_{\mathcal{A}^c} \hug[s]^{2^\star}\,{\rm d}t\,\ds=1+o(Q(s)),\\
		&&\int_{\mathcal{A}} \hug[s]^{2^\star-1}\hug\,{\rm d}t\,\ds=Q(s)-\int_{\mathcal{A}^c} \hug[s]^{2^\star-1}\hug\,{\rm d}t\,\ds = Q(s)+o(Q(s)),\\
		&&\int_{\mathcal{A}} \hug[s]^{2^\star-2}\hug^2\,{\rm d}t\,\ds=o(Q(s)).
	\end{eqnarray*}
	Similarly we estimate in the region $\mathcal{A}^c$.
	Adding the estimates over the partitions $\mathcal{A}$ and $\mathcal{A}^c$, we get $\|\hat{v}_s\|_{L^{2^{\star}}}^{2^{\star}} = 2 + 2Q(s) + o(Q(s)).$ This combined with $(2+2\eta)^{\frac{2}{2^\star}}=2^{\frac{2}{2^\star}}+\frac{2^{\frac{2}{2^\star}+1}}{2^\star}\eta+o(\eta)$ for $\eta >0$ small, we get the desired result.
	\medspace

\noindent
Proof of (c): We claim:
	\be\label{dist-est-1}
	\hat{\mathfrak{m}}(\hat{v}_s) = \left(1+Q(s) + o(Q(s))\right)^2, 
	\ee
	as $s\to\infty$. This together with part $(a)$ of this proposition gives the desired result
	\begin{eqnarray*}
		\di (\hat{v}_s,\hat{\mathcal{M}}_{\mbox{\tiny{HS}}})^2  &=& \|\hat{v}_s\|_{\varepsilon}^2 - S_{\gamma}\hat{\mathfrak{m}}(\hat{v}_s) \\
		&=& 2S_{\gamma} - S_{\gamma} \left(1+Q(s)+o(Q(s))\right)^2 + 2S_{\gamma} Q(s)\\
		&=& S_{\gamma} + o(Q(s)).
	\end{eqnarray*}	 
We divide proof the claim in several small steps.
	
	\noindent {\bf Step~1:} It follows from Lemma~\ref{dofm} that corresponding to each $s\in\R$ there exists $\tau\equiv \tau(s)\in\R$ such that 
	\be\nonumber
	\hat{\mathfrak{m}}(\hat{v}_s) = \left(\int_{\R\times\mathbb{S}^{N-1}}\hat{v}_s\hat{U}[\tau]^{2^{\star}-1}\,{\rm d}t\,{\rm d}\sigma\right)^2.
	\ee

	\medskip
	
	\noindent {\bf Step~2:} Let $s_n\to\infty$ as $n\to\infty$ and $\tau_n$ be the maximizer corresponding to $\hat{\mathfrak{m}}(\hat{v}_{s_n})$, then $|\tau_n|\to \infty$ and $|s_n-\tau_n|\to\infty$ can not happen simultaneously. 
	Otherwise, $\hat{\mathfrak{m}}(\hat{U}_{\gamma}[\tau_n])\to 0$ as $n\to\infty$ contradicting the fact that $\hat{\mathfrak{m}}(\hat{U}_{\gamma}[\tau_n])=1$ for each $n\in\mathbb{N}$.
	\medskip
	
	\noindent {\bf Step~3:} In this step we show that, either $\tau_n\to 0$ or $|\tau_n - s_n|\to 0$ as $n\to\infty$.
	First using the change of variable, we observe that
	\be\nonumber
	\int_{\R\times\mathbb{S}^{N-1}} \hat{U}_{\gamma}[s_n]\hat{U}_{\gamma}[\tau_n]^{2^{\star}-1}\,{\rm d}t\,{\rm d}\sigma = 	\int_{\R\times\mathbb{S}^{N-1}} \hat{U}_{\gamma}[s_n-\tau_n]\hat{U}_{\gamma}[0]^{2^{\star}-1}\,{\rm d}t\,{\rm d}\sigma.
	\ee
	Thus without loss of generality, let $\{\tau_n\}\subset \R$ be bounded and up to a subsequence (still denoted by $\tau_n$), $\tau_n\to\tau_0$ as $n\to\infty$. Then Br\'ezis-Lieb lemma yields,  $$\hat{U}_{\gamma}[\tau_n]\to\hat{U}_{\gamma}[\tau_0] \text{ in } L^{2^{\star}}(\R\times \mathbb{S}^{N-1}).$$
	Therefore 
	\begin{eqnarray*}
		1+ \int_{\R\times\mathbb{S}^{N-1}} \hat{U}_{\gamma}[s_n]\hat{U}_{\gamma}[0]^{2^{\star}-1}\,{\rm d}t\,{\rm d}\sigma &\leq& \int_{\R\times\mathbb{S}^{N-1}} \hat{U}_{\gamma}[s_n]\hat{U}_{\gamma}[\tau_n]^{2^{\star}-1}\,{\rm d}t\,{\rm d}\sigma\\
		&&\qquad + \int_{\R\times\mathbb{S}^{N-1}} \hat{U}_{\gamma}[0]\hat{U}_{\gamma}[\tau_n]^{2^{\star}-1}\,{\rm d}t\,{\rm d}\sigma\\
		&\leq& 1+o(1).
	\end{eqnarray*}
	Letting $n \rightarrow \infty$ we conclude
	\be\nonumber
	\int_{\R\times\mathbb{S}^{N-1}} \hat{U}_{\gamma}[0]\hat{U}_{\gamma}[\tau_0]^{2^{\star}-1}\,{\rm d}t\,{\rm d}\sigma =1,
	\ee
	and hence $\tau_0=0$ from the equality cases in H\"older's inequality. 
	
	\medskip
	
	\noindent {\bf Step~4:} Next we show that $\left\|\frac{\hat{U}_{\gamma}[\tau_n]}{\hat{U}_{\gamma}[0]}-1\right\|_{\infty} = o(1)$ when $s_n\to\infty$ and $\tau_n\to 0$ as $n\to\infty$.
	
By Mean-value theorem  for each $n\in\mathbb{N}$ and $t,$ there exists some $\xi_n\in \left(t-|\tau_n|,t+|\tau_n|\right)$ such that
	\begin{eqnarray*}
		\left|\frac{\hat{U}_{\gamma}[\tau_n](t)}{\hat{U}_{\gamma}[0](t)}-1\right| &=& \left|\frac{\sech^{\frac{N-2}{2}} \left(\theta(t+\tau_n)\right) - \sech^{\frac{N-2}{2}} \left(\theta t\right)}{\sech^{\frac{N-2}{2}} \left(\theta t\right)}\right|\\
		&\leq&\frac{N-2}{2}\left| \frac{\sech({\theta}\xi_n)}{\sech({\theta}t)}\right|^{\frac{N-2}{2}}\left|\tanh \left(\theta \xi_n \right)\right||\tau_n|\leq C|\tau_n| = o(1) 
	\end{eqnarray*}
where $C>0$ independent of $n$ and $t$.
 
 \medskip
 	
	\noindent {\bf Step~5:} In this step we show that $$\left(\hat{\mathfrak{m}}\left(\hat{U}_{\gamma} + \hat{U}_{\gamma}[s_n]\right)\right)^{\frac{1}{2}} = \left(1+Q(s_n)+o(Q(s_n))\right). $$ 
	We notice that
	\begin{eqnarray*}
		\hat{\mathfrak{m}}\left(\hat{v}_{s_n}\right)^{\frac{1}{2}} &=& \int_{\R\times \mathbb{S}^{N-1}} \left(\hat{U}_{\gamma}[s_n]^{2^{\star}-1}+\hat{U}_{\gamma}[0]^{2^{\star}-1}\right) \hat{U}_{\gamma}[\tau_n]\,{\rm d}t\,{\rm d}\sigma\\
		&=& \int_{\R\times \mathbb{S}^{N-1}} \hat{U}_{\gamma}[0]^{2^{\star}-1}\hat{U}_{\gamma}[\tau_n]\,{\rm d}t\,{\rm d}\sigma +  \int_{\R\times \mathbb{S}^{N-1}} \hat{U}_{\gamma}[s_n]^{2^{\star}-1}\hat{U}_{\gamma}[0]\,{\rm d}t\,{\rm d}\sigma\\
		&&\qquad+ \int_{\R\times \mathbb{S}^{N-1}} \hat{U}_{\gamma}[s_n]^{2^{\star}-1}\left(\hat{U}_{\gamma}[\tau_n]-\hat{U}_{\gamma}[0]\right)\,{\rm d}t\,{\rm d}\sigma\\
		&\leq& 1+ Q(s_n) + o(1) \int_{\R\times \mathbb{S}^{N-1}} \hat{U}_{\gamma}[s_n]^{2^{\star}-1}\hat{U}_{\gamma}[0]\,{\rm d}t\,{\rm d}\sigma\\
		&=& 1+Q(s_n)+ o(Q(s_n)).
	\end{eqnarray*}
	
	Moreover, \be\nonumber
	1+Q(s_n)= \int_{\R\times\sn}\hat{v}_{s_n} \hat{U}_{\gamma}^{2^{\star}-1}\,{\rm d}t\,{\rm d}\sigma \leq \hat{\mathfrak{m}}\left(\hat{U}_{\gamma} +\hat{U}_{\gamma}[s_n]\right)^{\frac{1}{2}} \leq 1+Q(s_n)+ o(Q(s_n)).
	\ee
	Therefore,
	\be\label{dist-est-2}
	\hat{\mathfrak{m}}\left(\hat{v}_{s_n}\right)^{\frac{1}{2}} = 1+Q(s_n)+o(Q(s_n)).
	\ee
Hence we have the desired claim.	
	
\end{proof}

\noindent	
{\bf Proof of Theorem~\ref{2-peak-th}}
\begin{proof} Using the results in Proposition~\ref{2-peak-key-est} we get
\begin{eqnarray*}
\frac{\delta(\hat{v}_s)}{\mbox{dist} (\hat{v}_s, \mathcal{M}_{\mbox{\tiny{HS}}})^2} &=& \frac{\|\hat{v}_{s}\|_{\varepsilon}^2-S_{\gamma}\|\hat{v}_{s}\|_{L^{2^{\star}}}^2}{\di\left(\hat{v}_{s},\hat{\mathcal{M}}_{\mbox{\tiny{HS}}}\right)^2}\\
&=& \frac{2S_{\gamma}+2S_{\gamma}Q(s)-S_{\gamma}\left(2^{\frac{2}{2^{\star}}}+ 2^{\frac{2}{2^{\star}}+1}Q(s) +o(Q(s))\right)}{S_{\gamma}+ o(Q(s))}\\
&<& 2-2^{\frac{2}{2^{\star}}},
\end{eqnarray*}
for $s>0$ large enough. This completes the proof.

\end{proof}
\section{Proof of Bianchi-Egnell extremizer}\label{proofs}

This section deals with the existence of the minimizer attaining $\cbe(\gamma).$ By Proposition \ref{constantcber}, $\cber(\gamma)$ is constant and by \cite[Proposition 4.1]{WeiWu-2024}   
\begin{align*}
\cbe(\gamma) \leq \cber(\gamma) < \min \ \{\cbe^{\mbox{\tiny{loc}}}(\gamma), \cbe^{\mbox{\tiny{2-peak}}}(\gamma)\} \ \ \ \mbox{for}\ \ \ \gamma\in\left[\gamma_c^{\star},\frac{(N-2)^2}{4}\right),
\end{align*}
as in either cases the test functions used are radial (radial third eigenfunction for $\cbel(\gamma)$ and sum of two weakly interacting HS-extremizer for $\cbe^{\mbox{\tiny{2-peak}}}(\gamma)$). By definition 
\eqref{thegamma}, the same strict inequality persists for $\gamma > \gamma_0.$

\medskip

Let $\{u_k\}$ be a minimizing sequence of $\cbe(\gamma)$, and normalized $\|u_k\|_{\lts}=1$. Then
\begin{align*}
\|u_k\|_{\ga}^2- S_{\gamma}=\cbe(\gamma) \left(\|u_k\|_{\ga}^2-\sg \m(u_k)\right)+o(1).
\end{align*} 
Simplifying and neglecting the positive quantity $\m(u_k)$, we have, 
\begin{align*}
\|u_k\|_{\gamma}^2\leq \frac{S_{\gamma}}{1-\cbe(\gamma)}+o(1),
\end{align*}
and hence $\{u_k\}$ is bounded in $H^1(\rn)$. The main difficulty is to extract a non-zero weak limit as 
 the Hardy-Sobolev inequality is not translation invariant. On the other hand, by concentration compactness principle of Lions \cite{LP1-1984, LP2-1984}, every bounded sequence in $H^1$, up to translation and dilation, admits a non-zero weak limit. For radial bounded sequence in $H^1$ though we can recover the non-zero weak limit by dilation, but in general, it is not true. It is our first goal of this sub-section to extract a non-zero weak limit.

\subsection{Non-zero weak limit}

We show that if $\cbe(\gamma)$ is strictly below the hidden critical level,  $\cbe(\ga)<1-\frac{S_{\ga}}{S}$, then a minimizing sequence for $\cbe(\ga)$ admits (up to dilations) a nonzero weak limit. For our convenience, in this section we work with the cylindrical co-ordinates. Recall for $\phi \in H^1(\R^N),$ we can lift to $\hat \phi \in H^1(\R \times \sn)$ according to the rule described in section \ref{isometric lifting}.
 
 We recall a few notations:  $\tau_s\hat{\phi}(t,w) \coloneqq \hat{\phi}(t-s,w)$ denotes the translation of $\hat{\phi}$ in the $\R$-variable. For $i \in \mathbb{N} \cup \{0\},$ let $D(i)$ be the dimension of spherical harmonics of degree $i$ in $\sn.$ 
 We fix a $L^2$-orthonormal basis $\{Y_j^i\}_{j=1}^{D(i)}$ for the spherical harmonics of degree $i$, and we let $P_{ij}$ be the $L^2$-projection of $H^1(\R \times \sn)$ onto the subspace $H^1(\R)Y_j^i$.
\begin{Lem}{\label{Non-zero weak limit}}
	Suppose $\cbe(\ga)<1-\frac{S_{\ga}}{S}$ holds, and let $\left\{\hat{\phi}_n\right\}\subset H^1(\R\times\sn)\setminus \hat{\mathcal{M}}_{\mbox{\tiny HS}}$ be a minimizing sequence for $\cbe(\ga)$ satisfying $\|\hat{\phi}_n\|_{\lts}=1$ for every $n\in \mathbb{N}$. 
	
	\medskip
	
	Then there exists a sequence $\{s_n\}\subset\R$ and $0\not\equiv\hat{\phi}\in H^1(\R\times\sn)$ such that (up to a subsequence) 
	\begin{align*}
	\tau_{s_n}\hat{\phi}_n\rightharpoonup \hat{\phi} \ \ \mbox{in} \ H^1(\R\times\sn).
	\end{align*}
\end{Lem}

\begin{proof}
	We break the proof into two steps. 
	By assumption $\frac{S_{\ga}}{S}< 1- \cbe (\ga),$ and therefore, a $k_0\in\mathbb{N}$ exists such that
	\begin{equation}\label{ES:1}
		\frac{S_{\ga}}{\left(1-\frac{\ga}{\frac{(N-2)^2}{4} +k_0(N-2+k_0)}\right)S} < 1-\cbe (\ga).
	\end{equation}

	\noindent
	{\bf Step 1:} There exists a $\delta>0$ and $1 \leq k \leq k_0$ such that
	\be\nonumber
	\liminf_{n\to\infty} \left\| P_{kj}\hat{\phi}_n\right\|_{\lts}^2 > \delta, \qquad\text{for some }j\in\{1,\cdots, D(k)\}.
	\ee
	 If we write $\hat{\phi}= \sum_{i=0}^{\infty}\sum_{j=0}^{D(i)}\phi_{ij}Y^{i}_{j}$ then by definition $P_{ij}\hat{\phi} = \phi_{ij}Y^i_j$. 
	Now assume that $\left\|P_{ij}\hat{\phi}_n\right\|_{\lts} \to 0$ as $n\to\infty$, for every $i\in\{1,\cdots, k_0\},\,j\in\{1,\cdots, D(i)\}$. We decompose $\hat \phi_n = \hat \varphi_n + \hat \psi_n,$ where $\hat \varphi_n = \sum_{i=0}^{k_0}\sum_{j=1}^{D(i)}P_{ij}\hat{\phi}_n.$ Then 
\begin{align*}
\lim_{n \rightarrow \infty} \|\hat \varphi_n\|_{\lts} = 0, \ \ \lim_{n \rightarrow \infty} \|\hat \psi_n\|_{\lts} = 1. 
\end{align*}

Then by improved Hardy's inequality applied to $\hat \psi_n$ (see Lemma~\ref{Imp-Har} in the appendix) we get
\begin{align}\label{ES:2}
\|\hat{\phi}_n\|_{\var}^2 \geq \|\hat{\psi}_n\|_{\var}^2 \geq \left(1-\frac{\ga}{\frac{(N-2)^2}{4} +k_0(N-2+k_0)}\right)\|\hat{\psi}_n\|_{H^1(\R\times\sn)}^2.
\end{align}

Since $\hat{\mathfrak{m}}(\hat{\phi}_n) = \hat{\mathfrak{m}}(P_{01}\hat{\phi}_n)$ and  since $\|P_{01}\hat{\phi}_n\|_{\lts}^2\to 0$ as $n\to\infty$ we get
	 \be\nonumber
	 \hat{\mathfrak{m}}\left(P_{01}\hat{\phi}_n\right) \leq \|P_{01}\hat{\phi}_n\|_{\lts}^2 \to 0,\quad \text{as }n\to\infty. 
	 \ee
	
	 Using the minimizing property of $\{\hat{\phi}_n\}$, we get
	 
		 \begin{eqnarray*}
	 	\cbe (\ga) + o(1) &=& \frac{\|\hat{\phi}_n\|_{\var}^2-S_{\ga}\|\hat{\phi}_n\|_{\lts}^2}{\|\hat{\phi}_n\|_{\var}^2-S_{\ga}\hat{\mathfrak{m}}\left(P_{01}\hat{\phi}_n\right)}\nonumber\\
	 	&=& 1- \frac{S_{\ga}\|\hat{\phi}_n\|_{\lts}^2}{\|\hat{\phi}_n\|_{\var}^2+o(1)} +o(1),
	\end{eqnarray*}
Therefore,	
\begin{align*}
	\qquad \frac{S_{\ga}}{\|\hat{\phi}_n\|_{\var}^2} = 1-\cbe(\ga) +o(1).
\end{align*}	 	
	 	
By \eqref{ES:2} and since $\|\hat{\psi}\|_{H^1(\R\times\sn)}=\|\nabla \psi\|_{2},$
using Sobolev inequality, we get
\begin{align*}	
\frac{S_{\ga}}{\left(1-\frac{\ga}{\frac{(N-2)^2}{4}+k_0(N-2+k_0)}\right)S\|\hat{\psi}_n\|_{\lts}^2} \geq 1-\cbe(\ga) +o(1).
\end{align*}
Since $\|\hat{\psi}_n\|_{\lts}\to 1$ therefore taking limit as $n\to\infty$ we get a contradiction to \eqref{ES:1}.
This completes the proof of Step~1.

\medspace

{\bf Step~2:} In this step we show that, up to a subsequence, there exists $\{s_n\}\subset \R$ such that $\tau_{s_n}\hat\phi_n\not\rightharpoonup 0$ as $n\to\infty$. By Step~1, without loss of generality, we assume 
\begin{align*}
	&\quad \liminf_{n\to\infty}\left\|P_{11}\hat{\phi}_n\right\|_{\lts}^2 \geq \delta_0>0.
\end{align*}
For simplicity of notations we denote $P_{11}\hat{\phi}_n = g_nY_1^1,$ and  therefore
\begin{align*}
 \left\| g_n\right\|_{L^{2^{\star}}(\R)}^2 \geq \frac{\delta_0}{\|Y^1_1\|_{L^2(\sn)}^2}.
\end{align*}

Hence, by  Lion's lemma \cite{LP1-1984, LP2-1984}  up to translation $\tau_{s_n}$ in $\R$ variable, $
\tau_{s_n}(g_n)\not\rightharpoonup 0$ in $H^1(\R)$.

Indeed, by one dimensional Gagliardo-Nirenberg-Sobolev inequality (see Dolbeault et.al \cite[equation (1.1), together with interpolation inequality]{DELL-2014}), for every $p>2$ and for a fixed $R>0$ and $g\in H^1(\R)$ 
\begin{align}\label{1dimGNS}
\|g\|_{L^p(B_R(s))}\leq C_R(p) \|g\|_{\lt(B_R(s))}^{\eta}\|g\|_{H^1(B_R(s))}^{1-\eta},
\end{align} 
where $\eta=\frac{2}{p(p-1)},$ $C_R(p)$ is a constant depending on $R$ and $p$ and independent of the centre $s$ and $B_R(s)$ is an open interval in $\R$ with radius $R$ and centre $s$. 

We apply \eqref{1dimGNS} for $p = 2^{\star}$ to $g_n$, and since $2^{\star}$ is subcritical for $N=1$ and adapting Lion's Lemma \cite{LP1-1984,LP2-1984}  (see also Willem \cite[Lemma 1.21]{W-1996}) we conclude  $\displaystyle\liminf_{k\to\infty} \displaystyle\sup_{s\in \R} \int_{B_R(s)}  g_n^2 > \delta^{\prime} >0$. Hence, there exist a sequence $\{s_n\}$ such that 
\begin{align*}
\int_{B_R(s_n)}   g_n^2\ {\rm d}t>\delta^{\prime}.  
\end{align*}
By abuse of notation we denote the translated sequence $\tau_{s_n}\hat{\phi}_n$ also by  $\hat \phi_n$. 
By weak continuity, we conclude the weak limit $\hat\phi \not\equiv 0,$ and  completes the proof.
\end{proof}	 

\medskip

Once we have non-zero weak limit, we can follow the proof of Tobias K\"onig \cite{KT-2025} or the CKN-version of Wei-Wu \cite{WeiWu-2024} to find the existence of minimizer.

 We now mimic K\"{o}nig's argument and decompose $u_k=v+f_k,$ then both $H^1$-norm and $L^{2^{\star}}$-norm asymptotically decomposes:
\begin{align*}
\|u_k\|_{\ga}^2 = \|v\|_{\ga}^2+ \|f_k\|_{\ga}^2 + o(1), \ \ \|u_k\|_{L^{2^{\star}}}^{2^{\star}} = \|v\|_{L^{2^{\star}}}^{2^{\star}} + \|f_k\|_{L^{2^{\star}}}^{2^{\star}} + o(1).
\end{align*}

\begin{Lem}
As $k\to\infty$, we have,
\begin{align*}
\m(u_k)=\max \ \{\m(v),\m(f_k)\}+o(1).
\end{align*}
In particular,
\begin{align*}
\di (u_k,\mhs)^2=\|v\|_{\ga}^2+\|f_k\|_{\ga}^2-\sg\max\{\m(v),\m(f_k)\}+o(1).
\end{align*}
\end{Lem}
\begin{proof}
The proof follows \cite{KT-2025} in verbatim, hence we omit it.
\end{proof}

\noindent
{\bf{Proof of Theorem \ref{main2}}} 
\begin{proof}
First we show $\cbe(\gamma) < 1 - \frac{S_{\gamma}}{S}.$
We recall that $\cbe(\gamma)\leq \cber(\gamma)$.
Assume first $N\geq 4$.
In this case, by Lemma \ref{cbel less hidden level}
\begin{align*}
\cbe(\gamma) \leq \cber(\gamma) \leq \Lambda(\gamma) < 1 - \frac{S_{\gamma}}{S},
\end{align*}
for all $\gamma_0 \leq \gamma < \frac{(N-2)^2}{4}$ with strict inequality if $\gamma_0 < \gamma.$ If $\cbe(\gamma_0) = \cber(\gamma_0)$ then there is nothing to prove. Therefore, with out loss of generality we can assume strict inequality  $\cbe(\gamma) < \Lambda(\gamma)$ holds for all $\gamma_0 \leq \gamma < \frac{(N-2)^2}{4}.$
Now consider $N=3$. Recall in this case $\gamma_0 = \left(1-\left(\frac{N}{N+4}\right)^{\frac{N}{N-1}}\right)\frac{(N-2)^2}{4},$ and again by lemma \ref{cbel less hidden level}, $\Lambda(\gamma)<1-\frac{S}{S_{\gamma}}$ holds for $\gamma\in \left[\gamma_0, \frac{(N-2)^2}{4}\right)$ and hence $\cbe(\gamma)<1-\frac{S_{\gamma}}{S}$ holds. By Lemma \ref{Non-zero weak limit} we conclude the weak limit $v \not\equiv 0.$ We now prove $f_k \to 0$ in $H^1(\R^N)$. If not then (up to a subsequence) $\|f_k\|_{\gamma}^2\geq \epsilon$. 
 \medskip
 
 \noindent
{\bf{Case 1:}}
\begin{align}\label{case1}
\m(v)\geq \m(f_k)+o(1).
\end{align}
Then $\di(u_k,\mhs)^2=\|v\|_{\ga}^2+\|f_k\|_{\ga}^2-S_{\ga}\m(v).$

\medskip

\noindent
{\bf{Subcase 1(a):}}
\begin{align*}
\lim_{k\to\infty}\|f_k\|_{\lts}\leq \|v\|_{\lts}.
\end{align*}
Multiplying $u_k$ by a non-zero constant, we  assume $\|v\|_{\lts}=1$. Then
\begin{align*}
\cbe(\gamma)+o(1)&=\frac{\|u_k\|_{\ga}^2-\sg\|u_k\|_{\lts}^2}{\di(u_k,\mhs)^2}\\
&=\frac{\|v\|_{\ga}^2-\sg\|v\|_{\lts}^2+\|f_k\|_{\ga}^2-\sg((1-\|f_k\|_{\lts}^{2^\star})^{\frac{2}{2^{\star}}}-1)}{\|v\|_{\ga}^2+\|f_k\|_{\ga}^2-\sg\m(v)}\\
&=: \frac{A+B}{C+D}.
\end{align*}
Now if 
\begin{align*}
\|v\|_{\ga}^2-\sg\m(v)\neq 0.
\end{align*}
Then by definition 
\begin{align*}
\cbe(\ga)\leq \frac{\|v\|_{\ga}^2-\sg\|v\|_{\lts}^2}{\|v\|_{\ga}^2-\sg\m(v)} =:\frac{A}{C}.
\end{align*}
Hence by an algebraic manipulation,
\begin{align}\label{fk0}
\frac{B}{D} \coloneqq \lim_{k \rightarrow \infty}\frac{\|f_k\|_{\ga}^2-\sg((1-\|f_k\|_{\lts}^{2^{\star}})^{\frac{2}{2^{\star}}}-1)}{\|f_k\|_{\ga}^2}\leq \cbe(\ga).
\end{align}

If $\|v\|_{\ga}^2-\sg\m(v)=0$, then using $\|v\|_{\ga}^2-\sg \|v\|_{\lts}^2\geq 0$, we see that \eqref{fk0} holds as well.
 Therefore, using the strict $2$-peak level, for large $k$, 
\begin{align*}
1-\frac{\sg \left((1+\|f_k\|_{\lts}^{2^\star})^{\frac{2}{2^{\star}}}-1\right)}{\sg[f_k]\|f_k\|_{\lts}^2}< 2-2^{\frac{2}{2^{\star}}},
\end{align*}
where we denote $S_{\ga}[f_k] = \|f_k\|^2_{\ga}/\|f_k\|_{\lts}^2$,  the Hardy-Sobolev ratio. Since $\eta\to \frac{(1+\eta^p)^{\frac{2}{p}}-1}{\eta^2}$ is increasing in $(0,\infty)$ for $p>2$,
and $\|f_k\|_{\lts}\leq 1$, we get
\begin{align*}
1-\frac{\sg(2^{\frac{2}{2^{\star}}}-1)}{\sg[f_k]}< 2-2^{\frac{2}{2^{\star}}}
\end{align*}
which gives the  contradiction $\sg[f_k]< \sg$.

\medskip

\noindent
{\bf{Subcase 1(b):}} Now we assume 
\begin{align*}
\lim_{k \rightarrow \infty}\|f_k\|_{\lts}\geq \|v\|_{\lts}.
\end{align*}
We define the following quantities:
\begin{align*}
\tilde{u}_k=\frac{u_k}{\|f_k\|_{\lts}},\  \ \tilde{v}_k=\frac{v}{\|f_k\|_{\lts}}, \ \ \tilde{f}_k=\frac{f_k}{\|f_k\|_{\lts}}.
\end{align*}
Again by Br\'ezis-Lieb lemma we have,
\begin{align*}
\cbe(\ga)+o(1)&=\frac{\|\tilde{u}_k\|_{\gamma}^2-\sg \|\tilde{u}_k\|_{\lts}^2}{\|\tilde{u}_k\|_{\gamma}^2-\sg\m(\tilde{u}_k)}\\
&=\frac{\|\tilde{v}_k\|_{\ga}^2+\|\tilde{f}_k\|_{\ga}^2-\sg(1+\|\tilde{v}_k\|_{\lts}^{2^{\star}})^{\frac{2}{2^{\star}}}}{\|\tilde{v}_k\|_{\ga}^2+\|\tilde{f}_k\|_{\ga}^2-\sg\m(\tilde{v}_k)}\\
&=\frac{\|\tilde{v}_k\|_{\ga}^2-\sg+\|\tilde{f}_k\|_{\ga}^2-\sg[(1+\|\tilde{v}_k\|_{\lts}^{2^{\star}})^{\frac{2}{2^{\star}}}-1]}{\|\tilde{v}_k\|_{\ga}^2+\|\tilde{f}_k\|_{\ga}^2-\sg\m(\tilde{v}_k)}.
\end{align*}
Now using the relation \eqref{case1}, we have 
$\|\tilde{f}_k\|_{\ga}^2-\sg \m(\tilde{v}_k) \leq \|\tilde{f}_k\|_{\ga}^2-\sg \m(\tilde{f}_k)+o(1)$,
hence 
\begin{align*}
\cbe(\ga)+o(1)\leq \frac{\|\tilde{f}_k\|_{\ga}^2-\sg+\|\tilde{v}_k\|_{\ga}^2-\sg[(1+\|\tilde{v}_k\|_{\lts}^{2^{\star}})^{\frac{2}{2^{\star}}}-1]}{\|\tilde{f}_k\|_{\ga}^2-\sg\m(\tilde{f}_k)+\|\tilde{v}_k\|_{\ga}^2}.
\end{align*}
Since
\begin{align*}
\cbe(\ga)\leq\frac{\|\tilde{g}_k\|_{\ga}^2-\sg}{\|\tilde{g}_k\|_{\ga}^2-\sg\m(\tilde{g}_k)}.
\end{align*}
if the denominator does not vanishes, proceeding as in {\bf Sub-case 1(a)},  we get a similar contradiction $\sg[v]<\sg$.

\medskip

\noindent
{\bf{Case 2:}}
\begin{align}\label{case1}
\m(v)\leq \m(f_k)+o(1).
\end{align}
It follows similarly by interchanging the role of $v$ and $f_k$. Hence we get the strong convergence. Actually, in this case, by scaling we can reduce to the first case.

\medskip

Finally, for $\gamma \in\left[\gamma_0,\frac{(N-2)^2}{4}\right)$ we have $\cbe(\gamma)<\Lambda (\gamma)$ which ensures $v\not\in \mhs$ and hence is a minimizer for $\cbe (\gamma)$. This completes the proof of the theorem.

  \end{proof}

\medskip

\section{Appendix} \label{appendix}
\subsection{Interaction estimates between the bubbles}

\begin{Lem}[Interaction estimates]\label{BI}
	Let $N\geq 3,$  fix non-negative real numbers $\eta_1,\,\eta_2$ such that  $\eta_1+\eta_2=2^{\star},$ and let $ \lambda\in(0,1].$ Then 
	\be\nonumber
	\int_{\rn} U_{\gamma}^{\eta_1} U_{\gamma}[\lambda]^{\eta_2}\,{\rm d}x =
	\begin{cases}
	\approx_{N,\gamma} \lambda^{\var\min\{\eta_1,\,\eta_2\}} \ \ \ \ \  \ \text{if }|\eta_1-\eta_2|>0,\\
	 \approx_{N,\gamma}	\lambda^{\var\frac{N}{N-2}}\log(\lambda^{-1})\,\, \ \ \ \text{if }\eta_1=\eta_2=\frac{2^{\star}}{2}.
	\end{cases}
	\ee
\end{Lem}

\begin{proof}
	On $B(0,\lambda^{-1}),\,|\lambda x|<1,$ and so
	\be\nonumber
	U_{\gamma}[\lambda](x)\approx_{N,\gamma} \lambda^{\varepsilon}|x|^{-\beta_{-}(\gamma)}.
	\ee
	On $B(0,\lambda^{-1})^c,$ $\frac{1}{|x|}\leq \lambda<1$ and therefore
	\begin{eqnarray}
	U_{\gamma}(x) &\approx_{N,\gamma}& \frac{1}{|x|^{\beta_{-}(\gamma)}\left(1+|x|^{2\varepsilon \frac{2}{N-2}}\right)^{\frac{N-2}{2}}}\nonumber\\
	&\approx_{N,\gamma}& \frac{1}{|x|^{\beta_{-}(\gamma)+2\varepsilon}\left(1+\frac{1}{|x|^{2\varepsilon \frac{2}{N-2}}}\right)^{\frac{N-2}{2}}}\approx_{N,\gamma}|x|^{-\beta_{+}(\gamma)}.\nonumber
	\end{eqnarray}
	On the other hand, on $B(0,\lambda^{-1})^c,$
	\begin{eqnarray}
	U_{\gamma}[\lambda](x) &\approx_{N,\gamma}& \frac{\lambda^{\frac{N-2}{2}}}{|\la x|^{\beta_{-}(\gamma)}\left(1+|\la x|^{2\varepsilon\frac{2}{N-2}}\right)^{\frac{N-2}{2}}}\nonumber\\
	&\approx_{N,\gamma}&\lambda^{\frac{N-2}{2}-\beta_{+}(\gamma)}|x|^{-\beta_{+}(\gamma)}\nonumber\\
	&\approx_{N,\gamma}&\la^{-\varepsilon}|x|^{-\beta_{+}(\gamma)}.\nonumber
	\end{eqnarray}
	Since we have $\eta_1+\eta_2 = 2^{\star},$ we get
	\begin{eqnarray}
	&\,&\int_{\R^{N}} U_{\gamma}(x)^{\eta_1} U_{\gamma}[\lambda](x)^{\eta_2}\,{\rm d}x =\int_{B(0,\la^{-1})}U_{\gamma}(x)^{\eta_1} U_{\gamma}[\lambda](x)^{\eta_2}\,{\rm d}x\nonumber\\
	&\,&\qquad\qquad\qquad\qquad\qquad\qquad +\int_{B(0,\lambda^{-1})^c}U_{\gamma}(x)^{\eta_1} U_{\gamma}[\lambda](x)^{\eta_2}\,{\rm d}x\nonumber\\
	&\approx_{N,\ga}& \int_{t=0}^{\lambda^{-1}}\frac{\lambda^{\varepsilon\eta_2}t^{N-1}}{t^{\beta_{-}(\gamma)\frac{2N}{N-2}}\left(1+t^{2\varepsilon\frac{2}{N-2}}\right)^{\frac{N-2}{2}\eta_1}}\,{\rm d}t+\int_{t=\la^{-1}}^{\infty}t^{N-1-\beta_{+}(\gamma)\frac{2N}{(N-2)}}\lambda^{-\var\eta_2}\,{\rm d}t\nonumber\\
	&\approx_{N,\gamma}& \lambda^{\varepsilon \eta_2}\int_{t=1}^{\lambda^{-1}}t^{N-1-\beta_{-}(\gamma)2^{\star}-2\varepsilon\eta_1}\,{\rm d}t + \lambda^{\varepsilon\eta_1}+\lambda^{\varepsilon \eta_2}.\nonumber
	\end{eqnarray}
	Thus whenever $\eta_1=\frac{N-\beta_{-}(\gamma)2^{\star}}{2\var}=\frac{N}{N-2}=\frac{2^{\star}}{2}$,  $\log$ term will appear.
	
	\begin{itemize}
		\item If $\eta_1\geq \eta_2+\delta'$ (for some $\delta'>0$), the above expression becomes comparable to $\lambda^{\var\eta_2}$.
		\item If $\eta_2\geq \eta_1+\delta'$ (for some $\delta'>0$), the above expression becomes comparable to $\la^{\var\eta_1}$.
		\item If $\eta_1=\eta_2=\frac{2^{\star}}{2}$, the above expression becomes comparable to $\lambda^{\var\frac{N}{N-2}}\log (\lambda^{-1})$.
	\end{itemize}
\end{proof}

\begin{Rem}\label{bub-int-cyl}
	In cylindrical co-ordinates $\hat{U}_{\ga}(t,\sigma) \coloneqq e^{-\frac{N-2}{2}t}U_{\ga}(e^{-t}\sigma)$ the interaction estimate translates to the following: for every $s >0$ and $\eta_1,\,\eta_2\in\R^{+}$ with $\eta_1+\eta_2=2^{\star}$,
		\begin{align*}
	\int_{\R\times\mathbb{S}^{N-1}}\hat{U}_{\gamma}^{\eta_1}\hat{U}_{\gamma}[s]^{\eta_2}\,{\rm d}t\,{\rm d}\sigma = 
	\begin{cases}
	\approx_{N,\gamma}	e^{-s\var\min\{\eta_1,\,\eta_2\}},\quad\text{if }|\eta_1-\eta_2|>0\\
	\approx_{N,\gamma}	s e^{-s\var \frac{N}{N-2}}, \ \ \qquad\quad\text{if }\eta_1=\eta_2=\frac{2^{\star}}{2},
	\end{cases}
		\end{align*}
\end{Rem}

\subsection{The equation satisfied by $h_{\alpha}$}
\begin{Lem}\label{O.D.E}
Let $h_{\alpha}(t)=\sech^{\frac{N-2}{2}}\alpha t$ then,
\begin{align*}
-h''_{\alpha}(t)+ \left( \frac{N-2}{2}\right)^2\alpha^2 h_{\alpha}(t)=\ca h_{\alpha}^{2^{\star}-1}(t)
\end{align*}
 where $\ca= \left[\left(\frac{N-2}{2}\right)^2+\frac{N-2}{2}\right]\alpha^2$.

\end{Lem}
\begin{proof}
We compute the first and second derivative of $h_{\alpha}(t)=\sech^{\frac{N-2}{2}}\alpha t$. Then,
\begin{align*}
	h_{\alpha}'(t)&=-\frac{N-2}{2}\alpha \sech^{\frac{N-2}{2}}\alpha t \tanh\alpha t=-\frac{N-2}{2}\alpha h_{\alpha}(t)\tanh\alpha t. \\
	h_{\alpha}''(t)&= -\frac{N-2}{2}\alpha h'_{\alpha}(t)\tanh\alpha t -\frac{N-2}{2}\alpha^2 h_{\alpha}(t)\sech^2\alpha t\\
&=\left(\frac{N-2}{2}\right)^2 \alpha^2 h_{\alpha}(t)\tanh^2\alpha t-\frac{N-2}{2}\alpha^2 h_{\alpha}(t)\sech^2\alpha t\\
&=\left(\frac{N-2}{2}\right)^2 \alpha^2 h_{\alpha}(t)-\left[\left(\frac{N-2}{2}\right)^2+\frac{N-2}{2}\right]\alpha^2 h_{\alpha}(t)h_{\alpha}^{2^{\star}-2}(t). 
\end{align*}
Therefore,
$$-h''_{\alpha}(t)+ \left( \frac{N-2}{2}\right)^2\alpha^2 h_{\alpha}(t)=\ca h_{\alpha}^{2^{\star}-2}(t)h_{\alpha}(t).$$
\end{proof}

\subsection{Spherical Harmonics Decomposition} 
\begin{Lem}
Let $f\in H^1(\R \times \sn)$. Then there exists $\{f_{ij}\}$ where $f_{ij}\in H^1(\R)$ 
\begin{align*}
f(t,\sigma)=\sum_{i=0}^{\infty}\sum_{j=1}^{D(i)}f_{ij}(t)Y_j^i(\sigma) \ \ \ \mbox{in} \ \ H^1(\R \times \sn),
\end{align*} 
where $D(i)$ is the dimension of the space of homogeneous harmonic polynomial of degree $i$ and $\{Y_j^i\}_{j = 1}^{D(i)}$ is an $\lt(\sn)$ orthonormal basis.   
\end{Lem}

\noindent
{\bf{Plancharel-Parseval's Identity:}}
\begin{Cor}\label{parseval's identity}
The following equalities hold
\begin{align*}
 \|f\|_{H^1(\R\times \sn)}^2&=\sum_{i=0}^{\infty}\sum_{j=1}^{D(i)}\intc\left(|f'_{ij}|^2 +i(i+N-2)|f_{ij}|^2+\frac{(N-2)^2}{4}f_{ij}^2\right)(Y_j^i)^2\ {\rm d}t\,\ds, \\ 
\|f\|_{L^2(\R\times \sn)}^2
&=\sum_{i=0}^{\infty}\sum_{j=1}^{D(i)}\intR|f_{ij}|^2\ {\rm d}t.
\end{align*}
\end{Cor}
Interested reader may consult the following references \cite{Muller-1966,SW-1971,SB-2015} for a proof.  
\subsection{Improved Hardy's inequality}
\begin{Lem}\label{Imp-Har}
	For $\hat{u}\in \left(\bigoplus_{i=0}^{k_0}\bigoplus_{j=0}^{D(i)}H^1(\R)Y^i_j\right)^{\perp}$ we have
	\begin{equation}\label{IHI}
		\left(\frac{(N-2)^2}{4}+k_0(N-2+k_0)\right) \|\hat{u}\|_{\lt(\R\times\sn)}^2 \leq  \| \hat u\|_{H^1(\R\times\sn)}^2.
	\end{equation}
\end{Lem}

\begin{proof}
	For $\hat{u}\in \left(\bigoplus_{i=0}^{k_0}\bigoplus_{j=0}^{D(i)}H^1(\R)Y^i_j\right)^{\perp}$, we write
	\be\label{ES:4}
	\hat{u}(t,\sigma) = \sum_{i=k_0+1}^{\infty}\sum_{j=0}^{D(i)} u_{ij}(t)Y^i_j(\sigma).
	\ee
Using the Parseval's identity (Corollary~\ref{parseval's identity}), we conclude 	
	\begin{align*}
	&\left\{\frac{(N-2)^2}{4}+k_0(N-2+k_0)\right\} \int_{\R\times\sn}\hat{u}^2\,{\rm d}t\,{\rm d}\sigma\\
		&\qquad\qquad\leq \int_{\R\times\sn} \left\{\left|\partial_t \hat u\right|^2 + |\nabla_{\sn}\hat{u}|^2+\frac{(N-2)^2}{4}\hat{u}^2\right\}{\rm d}t\,{\rm d}\sigma.
	\end{align*}
   This completes the proof. 
\end{proof}
\medskip

\noindent
{\bf Acknowledgement.}  S. Chakraborty acknowledges the support of the Tata Institute of Fundamental Research, Bengaluru for the Institute Postdoctoral Fellowship. D. Karmakar acknowledges the support of the Department of Atomic Energy, Government of India, under project no. 12-R\&D-TFR-5.01-0520. 
 We thank Tian Xingliang for kindly pointing out the reference \cite{Wei25} to us and providing us a copy of their work.

\medskip

\noindent
{\bf Competing interests.} The authors have no competing interests to declare that are relevant to the content of this article.

\medskip

\noindent
{\bf Data availability statement.} Data sharing not applicable to this article as no datasets were generated or analysed during the current study.

\bibliographystyle{alpha}
\bibliography{HSBE.bib}

\end{document}